\documentclass[11pt]{amsart}
\usepackage[margin=1.15in]{geometry}
\usepackage{amsmath,amssymb,amsthm,mathtools}
\usepackage{booktabs}
\usepackage{enumitem}
\usepackage{microtype}
\usepackage{xcolor}
\usepackage{tikz}
\usetikzlibrary{arrows.meta,decorations.markings,positioning,calc}
\usepackage[colorlinks=true,linkcolor=blue!50!black,citecolor=blue!50!black,urlcolor=blue!50!black]{hyperref}
\usepackage[capitalise,nameinlink]{cleveref}
\hypersetup{pdftitle={Nim heaps in the Game of Cycles: the parity conjecture fails on trees},pdfauthor={Muchen Xu},pdfkeywords={Game of Cycles, Game of Arrows, impartial games, Sprague-Grundy theory, trees, spiders}}

\newtheorem{theorem}{Theorem}[section]
\newtheorem{lemma}[theorem]{Lemma}
\newtheorem{proposition}[theorem]{Proposition}
\newtheorem{corollary}[theorem]{Corollary}
\newtheorem{conjecture}[theorem]{Conjecture}
\newtheorem{question}[theorem]{Question}
\newtheorem{observation}[theorem]{Observation}
\newtheorem*{mainthm}{Main Theorem}
\theoremstyle{definition}
\newtheorem{definition}[theorem]{Definition}

\theoremstyle{remark}
\newtheorem{remark}[theorem]{Remark}

\newcommand{\G}{\mathcal{G}}
\newcommand{\gv}{\operatorname{g}}           
\newcommand{\mex}{\operatorname{mex}}
\newcommand{\m}{\operatorname{m}}
\newcommand{\sn}{\mathsf{n}}
\newcommand{\si}{\mathsf{i}}
\newcommand{\so}{\mathsf{o}}
\newcommand{\sff}{\mathsf{f}}
\newcommand{\R}{\mathrm{R}}
\newcommand{\Sp}{S}
\newcommand{\free}{\ast}
\newcommand{\rgt}{\rightarrow}
\newcommand{\lft}{\leftarrow}
\newcommand{\NN}{\mathbb{N}}
\newcommand{\IN}{\textup{\textsc{in}}}
\newcommand{\OUT}{\textup{\textsc{out}}}

\tikzset{
  v/.style={circle,fill=black,inner sep=0pt,minimum size=4.2pt},
  sp/.style={rectangle,fill=black,inner sep=0pt,minimum size=5pt},
  lf/.style={circle,draw=black,fill=white,inner sep=0pt,minimum size=4.2pt},
  ctr/.style={circle,draw=black,fill=red!70!black,inner sep=0pt,minimum size=6pt},
  pend/.style={densely dashed,gray},
  arc/.style={postaction={decorate},decoration={markings,mark=at position 0.62 with {\arrow{Stealth[length=5pt]}}},thick},
}

\title[Nim heaps in the Game of Cycles]{Nim heaps in the Game of Cycles:\\ the parity conjecture fails on trees}

\author{Muchen Xu}
\address{University of Southern California, Los Angeles, CA, USA}
\email{xupeter@usc.edu}

\date{October 2026}

\subjclass[2020]{Primary 91A46; Secondary 05C05, 05C57, 91A43}
\keywords{Game of Cycles, Game of Arrows, impartial games, Sprague--Grundy theory, trees, spiders}

\begin{document}

\begin{abstract}
The Game of Cycles, invented by Francis Su, is a two-player game in which the players take turns drawing arrows on the edges of a plane graph, never creating a sink or a source. On a tree only the sink--source rule matters, and the player who makes the last move wins. Several authors conjectured that on a tree the second player wins exactly when the number of markable edges is even. We show that this parity conjecture is false, and we explain why.

The explanation is a hidden game of Nim: a path segment that is blocked by an arrow at one end and free at the other behaves exactly like a nim heap whose size is its length. Combined with a calculus for cutting and gluing trees at vertices next to leaves, this shows that the conjecture fails as soon as a single tree has Grundy value at least $2$. The smallest counterexamples have $15$ vertices. Trees with Grundy values $2$ and $4$ first appear at $17$ vertices, and every value up to $63$ occurs, which answers a question of Fokkink and Zandee. We also refute a related conjecture of Zandee.

We then look closely at spiders. The spiders $\Sp(2,3,z)$ violate the conjecture for every $z\ge 14$. The Grundy values of $\Sp(1,2j,z)$ are given by an exact formula governed by binary carries. For $99$ pairs $(x,y)$, the value of $\Sp(x,y,z)$ is eventually the parity of $z$. On the positive side, the conjecture holds for all trees in which every thread has at most two edges, and this bound is sharp. Along the way we prove a twin-branch reduction, give an algorithm that is polynomial for spiders with a fixed number of legs, and collect the questions that we could not answer. The boldest asks whether every non-negative integer is the Grundy value of some tree.
\end{abstract}

\maketitle

\section{Introduction}\label{sec:intro}

The \emph{Game of Cycles} was invented by Francis Su for his book \emph{Mathematics for Human Flourishing}~\cite{Su2020}, and was first analysed by Alvarado et al.~\cite{Alvarado2021}. Two players take turns marking an unmarked edge of a plane graph with an arrow. One rule makes the game interesting: no vertex may ever become a \emph{sink}, with all of its edges marked and pointing in, or a \emph{source}, with all of its edges marked and pointing out. A player who completes a \emph{cycle cell}, a bounded face whose boundary arrows all point the same way around it, wins at once. Otherwise the player who makes the last legal move wins.

It is a delightful game to show to someone. The rules fit on a napkin, a game on a small board takes a minute, and yet good play is far from obvious. It has been studied on cycles and symmetric graphs~\cite{Alvarado2021}, three-legged spiders~\cite{Mathews2022}, grids and theta graphs~\cite{Barua2022} and cacti~\cite{Cacti2024}, as well as in variants~\cite{Sahai2023}. Two theses examine it~\cite{Lin2021,Zandee2022}, and a computer program decides small boards~\cite{Garcia2025}.

Anyone who plays a few games is tempted by a guess. On every board that Alvarado et al.\ solved, the first player wins exactly when the number of edges is odd. Lin~\cite[Conjecture~4.1]{Lin2021} refined this to the number of \emph{markable} edges. Boards with triangles turned out to break the rule~\cite{Cacti2024,Fokkink2024}, but on trees the guess survived. A tree has no cells, so only the sink--source rule matters, and the game becomes a pure ``last move wins'' impartial game, which Mathews~\cite{Mathews2022} called the \emph{Game of Arrows}. An edge at a leaf can never be marked, since marking it would turn the leaf into a sink or a source. So the natural parameter of a tree $T$ is
\[
  \m(T)=\#\{\text{edges of $T$ not incident to a leaf}\}.
\]

\begin{conjecture}[Parity Conjecture for trees; Lin~\cite{Lin2021}, Zandee~\cite{Zandee2022}, Fokkink--Zandee~\cite{Fokkink2024}]\label{conj:parity}
For every tree $T$, the second player wins the Game of Arrows on $T$ if and only if $\m(T)$ is even. Equivalently, the Grundy value of $T$ is $0$ exactly when $\m(T)$ is even.
\end{conjecture}

The evidence was good. The conjecture is known in several cases:
\begin{itemize}[leftmargin=*]
\item paths~\cite{Lin2021,Zandee2022,Sahai2023};
\item trees that have no vertex of degree~$2$ once the pendant edges are removed~\cite[Theorem~1]{Fokkink2024};
\item trees in which every vertex of degree~$2$ is adjacent to a leaf~\cite[Theorem~4.2]{Zandee2022};
\item three-legged spiders with two equal legs~\cite[Theorem~5.3]{Lin2021};
\item three-legged spiders all of whose legs have an even number of markable edges~\cite{Mathews2022}.
\end{itemize}
Fokkink and Zandee single out spiders as the difficult case. They also ask whether any graph has Grundy value $4$ or more, the largest value they had found being $3$~\cite{Fokkink2024}. Lin asks for a ``simple and nice winning strategy'' on general three-legged spiders~\cite[Section~7.1]{Lin2021}.

\subsection*{A hidden game of Nim}
The conjecture holds for all $5446$ trees with $2$ to $14$ vertices. On $15$ vertices it fails, exactly twice, for the spiders $\Sp(1,3,7)$ and $\Sp(1,4,6)$ of \cref{fig:cex}. When the first counterexample appeared on the screen, our first reaction was suspicion, so we checked it with two independently written brute-force programs. Our second reaction was curiosity: \emph{why} does parity fail here, after holding so stubbornly on thousands of smaller trees?

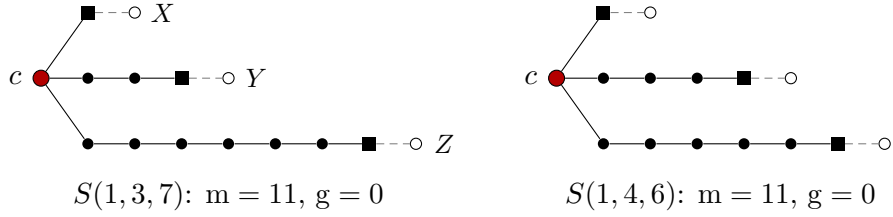
\begin{figure}[t]
\centering
\begin{tikzpicture}[scale=0.62]
\begin{scope}
\node[ctr] (c) at (0,0) {};
\node at (-0.55,0) {$c$};
\node[sp] (a1) at (1,1.4) {}; \node[lf] (a2) at (2,1.4) {};
\draw (c)--(a1); \draw[pend] (a1)--(a2);
\node[v] (b1) at (1,0) {}; \node[v] (b2) at (2,0) {}; \node[sp] (b3) at (3,0) {}; \node[lf] (b4) at (4,0) {};
\draw (c)--(b1)--(b2)--(b3); \draw[pend] (b3)--(b4);
\node[v] (c1) at (1,-1.4) {};
\foreach \i in {2,...,6} {\node[v] (c\i) at (\i,-1.4) {};}
\node[sp] (c7) at (7,-1.4) {}; \node[lf] (c8) at (8,-1.4) {};
\draw (c)--(c1); \foreach \i [evaluate=\i as \j using int(\i+1)] in {1,...,6} {\draw (c\i)--(c\j);}
\draw[pend] (c7)--(c8);
\node at (2.6,1.4) {\small $X$}; \node at (4.6,0) {\small $Y$}; \node at (8.6,-1.4) {\small $Z$};
\node at (4,-2.5) {$\Sp(1,3,7)$: $\m=11$, $\gv=0$};
\end{scope}
\begin{scope}[xshift=11cm]
\node[ctr] (c) at (0,0) {};
\node at (-0.55,0) {$c$};
\node[sp] (a1) at (1,1.4) {}; \node[lf] (a2) at (2,1.4) {};
\draw (c)--(a1); \draw[pend] (a1)--(a2);
\foreach \i in {1,2,3} {\node[v] (b\i) at (\i,0) {};}
\node[sp] (b4) at (4,0) {}; \node[lf] (b5) at (5,0) {};
\draw (c)--(b1)--(b2)--(b3)--(b4); \draw[pend] (b4)--(b5);
\foreach \i in {1,...,5} {\node[v] (c\i) at (\i,-1.4) {};}
\node[sp] (c6) at (6,-1.4) {}; \node[lf] (c7) at (7,-1.4) {};
\draw (c)--(c1); \foreach \i [evaluate=\i as \j using int(\i+1)] in {1,...,5} {\draw (c\i)--(c\j);}
\draw[pend] (c6)--(c7);
\node at (3.5,-2.5) {$\Sp(1,4,6)$: $\m=11$, $\gv=0$};
\end{scope}
\end{tikzpicture}
\caption{The two smallest counterexamples to the parity conjecture ($15$ vertices each). Dashed edges are pendant (never markable), squares are special vertices (non-leaves adjacent to a leaf), and the red vertex is the centre $c$. The legs of $\Sp(1,3,7)$ are named $X,Y,Z$ in \cref{tab:firstmoves}.}
\label{fig:cex}
\end{figure}

The answer is the heart of this paper, and it is pleasantly simple. Cut a leg of a spider somewhere in the middle. The part that falls off is a path whose inner end is now blocked by the new arrow and whose outer end is still free. Such a blocked path with $k$ edges behaves \emph{exactly like a nim heap of size $k$} (\cref{thm:runs}). So the Game of Arrows secretly contains Nim, with heaps of every size, and there is no reason for parity to decide anything. On $\Sp(1,3,7)$ every first move leads to a position with Grundy value $2$, $4$, $6$ or $8$. In each case the second player has an answer, typically a cut that removes a heap of exactly the right size, and the first player loses (\cref{tab:firstmoves}). On small trees the heaps are too short to matter, which is why the conjecture survives so long.

Once the nim heaps are visible, many things become possible. We can glue trees together and add their Grundy values, build trees with large Grundy values, prove that whole families of spiders violate the conjecture, and also prove that parity \emph{does} win on trees without long paths of degree-$2$ vertices.

\subsection*{Main results}
We use standard Sprague--Grundy theory, recalled in \cref{sec:prelim}, and write $\gv(T)$ for the Grundy value of the empty board on~$T$.

\begin{mainthm}
\leavevmode
\begin{enumerate}[label=\textup{(\Alph*)},leftmargin=*]
\item \textup{(Runs are nim heaps; \cref{thm:runs}.)} Consider a path segment of $k$ unmarked edges whose interior vertices have degree~$2$. If one end of the segment is free and the other end is blocked by an arrow, the segment is equivalent to a nim heap of size~$k$. The three other boundary types have Grundy values in $\{0,1\}$, given by explicit formulas.
\item \textup{(Splitting, gluing and reformulation; \cref{sec:consequences}.)} The game on $T$ is the disjunctive sum of the games on the \emph{pieces} obtained by cutting $T$ at the vertices adjacent to leaves. Gluing two trees at such vertices nim-adds their Grundy values. Consequently, \cref{conj:parity} is equivalent to $\gv(T)=\m(T)\bmod 2$ for every tree $T$, and it fails as soon as a single tree has Grundy value at least~$2$.
\item \textup{(Twin branches; \cref{thm:twin}.)} Suppose two branches at a vertex $c$ are isomorphic by a map fixing $c$. Then replacing both by a single pendant edge at $c$ changes neither the Grundy value nor the parity of $\m$. This recovers and extends the results for spiders with two equal legs.
\item \textup{(Counterexamples; \cref{thm:census}.)}
\begin{itemize}[leftmargin=*]
\item The parity conjecture holds for every tree with at most $14$ vertices.
\item On $15$ vertices there are exactly two counterexamples, the spiders $\Sp(1,3,7)$ and $\Sp(1,4,6)$, whose legs have $1,3,7$ and $1,4,6$ markable edges. Both have Grundy value $0$ but $11$ markable edges.
\item On $17$ vertices there are trees with Grundy values $2$ and $4$, and both are counterexamples.
\end{itemize}
\item \textup{(Large Grundy values; \cref{thm:allvalues}.)} Every integer $0\le g\le 63$ is the Grundy value of a tree; for example, the spider $\Sp(17,90,95)$ has Grundy value $32$. This answers the question of Fokkink and Zandee in a strong form, already within the class of trees.
\item \textup{(Zandee's second conjecture; \cref{prop:zandee}.)} Zandee conjectured that a tree with exactly one marked edge has Grundy value of the same parity as the number of remaining markable edges~\cite[Conjecture~4.3]{Zandee2022}. This fails already on $12$ vertices.
\item \textup{(Spider families; \cref{thm:s23,thm:onetwoj,thm:cert}.)}
\begin{itemize}[leftmargin=*]
\item $\gv(\Sp(2,3,z))=z\bmod 2$ for all $z\ge14$, so every such spider violates the parity conjecture.
\item For $j\ge2$ and $z\ge2j+2$, $\gv(\Sp(1,2j,z))$ is given by an explicit formula involving binary carries. It is periodic in $z$ with least period $2^{\lceil\log_2 j\rceil+2}$, and it violates the conjecture for a positive proportion of~$z$.
\item For $99$ pairs $(x,y)$, including $(2,3)$, $\gv(\Sp(x,y,z))=z\bmod 2$ for all $z$ beyond an explicit threshold.
\end{itemize}
\item \textup{(Short threads; \cref{cor:threads2,thm:threads3}.)} If every \emph{thread} of $T$ (a maximal path in a piece whose interior vertices have degree~$2$) has at most two edges, then $\gv(T)=\m(T)\bmod 2$. This contains the theorems of Fokkink--Zandee and of Zandee on trees without long paths of degree-$2$ vertices. The bound is sharp: there is a counterexample on $23$ vertices whose threads have at most three edges.
\end{enumerate}
\end{mainthm}

\subsection*{A word on method}
Much of this work was experimental in the best sense: compute, stare at the tables, guess, and then try to prove the guess. The proofs of (A)--(C), of the theorem on threads in (H), and of the formula for $\Sp(1,2j,z)$ in (G) are by hand. The result on $\Sp(2,3,z)$ is proved by hand apart from a finite check of $52$ explicit integer sequences. Parts (D)--(F), the rest of (G) and the sharpness example in (H) are computer-assisted, and we took care to make the computations trustworthy.
\begin{itemize}[leftmargin=*]
\item All statements about trees with at most $20$ vertices were confirmed by a program that implements the raw rules of the game and uses none of our theory. For at most $16$ vertices, two such programs, written independently, were used. All of these statements were also confirmed by the theory-based solver of \cref{prop:sum}.
\item The statements about larger spiders were confirmed by at least two independently written solvers based on \cref{prop:spiderrec}.
\item Statements about larger trees that are not spiders, such as the $H$-shaped trees with large Grundy values in \cref{sec:open}, rest on theory-based solvers, each confirmed by two independent implementations.
\end{itemize}
\Cref{tab:verification} lists which programs confirmed which claim, and all programs are available as ancillary files.

\subsection*{Organisation}
\Cref{sec:prelim,sec:forest} set up the game and a convenient bookkeeping of positions. \Cref{sec:runs} contains the nim-heap computation, and \cref{sec:consequences} its structural consequences: gluing, the reformulation, and twin branches. \Cref{sec:spiders} describes all positions of a spider and gives a polynomial-time algorithm, which partly addresses Lin's question. \Cref{sec:computation} reports the experiments. \Cref{sec:s23,sec:onetwoj} prove the results on spider families, and \cref{sec:threads} the theorem on short threads. We close in \cref{sec:open} with the questions we enjoyed most and could not answer. The boldest of them asks whether every non-negative integer is the Grundy value of some tree.

\section{Preliminaries}\label{sec:prelim}

\subsection{Impartial games}
We use standard facts from combinatorial game theory; see~\cite{WinningWays,LessonsInPlay,Siegel2013}. A finite impartial game under the normal play convention (the player unable to move loses) has a \emph{Grundy value}
\[
  \gv(P)=\mex\{\gv(P') : P' \text{ an option of } P\},
\]
where $\mex S$ is the least non-negative integer not in $S$. The second player wins from $P$ if and only if $\gv(P)=0$. In a disjunctive sum, a move is a move in exactly one component. The Sprague--Grundy theorem~\cite{Sprague1935,Grundy1939} gives $\gv(P+Q)=\gv(P)\oplus\gv(Q)$, where $\oplus$ is bitwise exclusive or (nim-sum). A \emph{nim heap} of size $k$ has Grundy value $k$. Two games $P,Q$ satisfy $\gv(P)=\gv(Q)$ if and only if $P+Q$ is a second-player win. This is the basis of the \emph{copycat} argument in \cref{thm:twin}.

\subsection{The Game of Arrows}
Let $T$ be a finite tree. A \emph{position} assigns to each edge either the label \emph{unmarked} or one of its two orientations (an \emph{arc}). A vertex $v$ is a \emph{sink} (resp.\ \emph{source}) if $v$ has at least one edge, all edges at $v$ are marked, and all of them point to (resp.\ away from) $v$. A \emph{move} orients one unmarked edge, and it is \emph{legal} if afterwards neither endpoint of that edge is a sink or a source. We start from the empty position, and the player who makes the last legal move wins. We write $\gv(T)$ for the Grundy value of the empty position.

A \emph{leaf} is a vertex of degree~$1$, and an edge at a leaf is \emph{pendant}. A vertex is \emph{special} if it is not a leaf but is adjacent to a leaf.\footnote{This is the ``special vertex'' of Fokkink and Zandee~\cite{Fokkink2024}, created by trimming pendant edges.}

\begin{lemma}\label{lem:pendant}
\begin{enumerate}[label=\textup{(\alph*)}]
\item A pendant edge can never be marked.
\item A special vertex can never become a sink or a source.
\end{enumerate}
\end{lemma}
\begin{proof}
Marking a pendant edge makes its leaf a source or a sink. Hence the pendant edge at a special vertex stays unmarked forever, so that vertex always keeps an unmarked edge.
\end{proof}

So the \emph{markable} edges are the non-pendant ones, and $\m(T)$ counts them. \Cref{fig:rules} illustrates the rules.

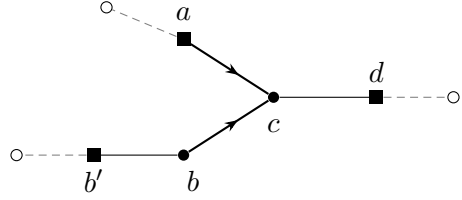
\begin{figure}[t]
\centering
\begin{tikzpicture}[scale=0.85]
\node[v] (c) at (0,0) {}; \node at (0,-0.45) {$c$};
\node[sp] (a1) at (-1.4,0.9) {}; \node at (-1.4,1.3) {$a$};
\node[lf] (a2) at (-2.6,1.4) {};
\node[v] (b1) at (-1.4,-0.9) {}; \node at (-1.25,-1.3) {$b$};
\node[sp] (b2) at (-2.8,-0.9) {}; \node at (-2.8,-1.3) {$b'$}; \node[lf] (b3) at (-4,-0.9) {};
\node[sp] (d1) at (1.6,0) {}; \node at (1.6,0.4) {$d$};
\node[lf] (d2) at (2.8,0) {};
\draw[arc] (a1)--(c);
\draw[arc] (b1)--(c);
\draw (c)--(d1);
\draw (b1)--(b2);
\draw[pend] (a1)--(a2);
\draw[pend] (b2)--(b3);
\draw[pend] (d1)--(d2);
\end{tikzpicture}
\caption{A position on a tree with $\m=4$. Dashed edges are pendant and can never be marked; squares are special vertices. The edge $cd$ may be oriented $c\to d$ but not $d\to c$, since then $c$ would become a sink. The edge $bb'$ may only be oriented $b'\to b$, since $b\to b'$ would make $b$ a source. The special vertices $a,b',d$ can never become sinks or sources.}
\label{fig:rules}
\end{figure}

A \emph{spider} is a tree with exactly one vertex of degree at least~$3$, its \emph{centre}. For integers $k\ge3$ and $x_1,\dots,x_k\ge1$, we write $\Sp(x_1,\dots,x_k)$ for the spider whose $t$-th leg is a path of $x_t+1$ edges. The last edge of each leg is pendant, so the leg has exactly $x_t$ markable edges and $\m(\Sp(x_1,\dots,x_k))=\sum_t x_t$. The spider has $1+\sum_t (x_t+1)$ vertices. For example, $\Sp(1,3,7)$ has legs of $2,4,8$ edges and $15$ vertices (\cref{fig:cex}). A leg consisting of a single pendant edge would make the centre special, and then the spider splits into paths (see \cref{cor:leafbranch}); this is why we assume $x_t\ge 1$.

\section{Positions as labelled forests}\label{sec:forest}

Whether a move is legal depends only on what is happening at its two endpoints. Before we can see any nim heaps, we need a good way to record this, and a way to break a position into independent pieces.

\begin{definition}\label{def:status}
Let $P$ be a position on $T$ and $v$ a vertex. Let $d(v)$ be the number of unmarked \emph{non-pendant} edges at $v$. The \emph{status} $s(v)\in\{\sn,\si,\so,\sff\}$ is defined as follows.
\begin{itemize}[leftmargin=*]
\item $s(v)=\sff$ (\emph{free}) if $v$ is special, or if $v$ has both an incoming and an outgoing arc.
\item Otherwise, $s(v)=\si$ if $v$ has at least one incoming arc, $s(v)=\so$ if $v$ has at least one outgoing arc, and $s(v)=\sn$ if $v$ has no arcs.
\end{itemize}
\end{definition}

\begin{lemma}[Legality]\label{lem:legal}
Let $uv$ be an unmarked non-pendant edge. Orienting it as $u\to v$ is legal if and only if
\[
  \bigl(d(u)\ge 2 \text{ or } s(u)\in\{\si,\sff\}\bigr)\quad\text{and}\quad\bigl(d(v)\ge 2\text{ or } s(v)\in\{\so,\sff\}\bigr).
\]
After the move, $d(u)$ and $d(v)$ decrease by one. The statuses update by $s(u)\gets s(u)\cdot\so$ and $s(v)\gets s(v)\cdot\si$, where
\[
\sn\cdot\so=\so,\quad \si\cdot\so=\sff,\quad \so\cdot\so=\so,\quad \sff\cdot\so=\sff,\qquad
\sn\cdot\si=\si,\quad \si\cdot\si=\si,\quad \so\cdot\si=\sff,\quad \sff\cdot\si=\sff.
\]
\end{lemma}

\begin{proof}
Only $u$ and $v$ change. After the move, $u$ has an outgoing arc, so it cannot be a sink; it is a source exactly when all of its edges are marked and none points into~$u$. If $u$ is special this never happens (\cref{lem:pendant}), and $s(u)=\sff$. Otherwise every edge at $u$ is non-pendant. Then all edges at $u$ are marked after the move if and only if $d(u)=1$. Moreover, none points into $u$ if and only if $u$ had no incoming arc, i.e.\ $s(u)\in\{\sn,\so\}$. The argument for $v$ is symmetric.
\end{proof}

Thus the game from $P$ depends only on the \emph{reduced position}: the forest $F(P)$ of unmarked non-pendant edges, in which every vertex carries its status. We formalise this as a game in its own right.

\begin{definition}\label{def:labelledgame}
A \emph{status-labelled forest} is a finite forest $F$ together with a map $s\colon V(F)\to\{\sn,\si,\so,\sff\}$. The \emph{arrow game} on $(F,s)$ has as moves the orientations $u\to v$ of edges $uv$ of $F$ that satisfy the condition of \cref{lem:legal}, where $d$ is the degree in $F$. Such a move deletes the edge $uv$ and updates $s(u)$ and $s(v)$ as in \cref{lem:legal}.
\end{definition}

By \cref{lem:legal}, the game from a position $P$ on $T$ is isomorphic to the arrow game on $(F(P),s)$.

\begin{lemma}[Dead edges]\label{lem:dead}
Suppose an edge $e=uv$ of a status-labelled forest has no legal orientation. Then it has none after any sequence of moves. Consequently, deleting $e$ and setting $s(u)=s(v)=\sff$ does not change the game.
\end{lemma}

\begin{proof}
Each orientation of $e$ is illegal because some endpoint $w$ has $d(w)=1$ and an unsuitable status. For such $w$, the edge $e$ is its only edge, so no move other than one on $e$ touches $w$. Hence $d(w)$ and $s(w)$ are frozen, and $e$ stays illegal forever. Therefore $e$ is never deleted, so the condition at $u$ (and at $v$) in \cref{lem:legal} is satisfied for every other edge at $u$ (and at $v$), because $d\ge2$ there whenever another edge exists. This is exactly the effect of deleting $e$ and giving $u,v$ the status $\sff$, which also never changes.
\end{proof}

\begin{lemma}[Splitting]\label{lem:split}
If $s(v)=\sff$, the game is unchanged when $v$ is replaced by $\deg(v)$ new vertices of status $\sff$, each incident with exactly one of the edges formerly at $v$.
\end{lemma}

\begin{proof}
By \cref{lem:legal}, the conditions at a vertex of status $\sff$ always hold, and the status $\sff$ never changes.
\end{proof}

\begin{definition}\label{def:normalform}
The \emph{normal form} of a status-labelled forest is obtained by deleting dead edges (setting their endpoints to $\sff$), splitting every vertex of status $\sff$, and discarding isolated vertices. Its connected components, each regarded as an arrow game, are the \emph{components}. The components of a position $P$ are those of $(F(P),s)$.
\end{definition}

\begin{proposition}\label{prop:sum}
Every position is the disjunctive sum of its components, so $\gv(P)=\bigoplus_K \gv(K)$.
\end{proposition}

\begin{proof}
By \cref{lem:dead,lem:split}, the normal form defines the same game as $(F(P),s)$. Components share no vertices, and a move changes only the statuses of the endpoints of the deleted edge, which lie in one component. So legality in one component never depends on another.
\end{proof}

In the empty position, the status of a vertex is $\sff$ if it is special and $\sn$ otherwise, and there are no dead edges.\footnote{An edge $uv$ with $d(u)=1$ and $s(u)=\sn$ would force $u$ to have no other edges at all, since otherwise $u$ would be special. Then $u$ would be a leaf and $uv$ pendant.} We call the components of the empty position the \emph{pieces} of $T$. They are obtained by deleting all pendant edges and cutting the tree at every special vertex. Every leaf of a piece has status $\sff$, and every other vertex has status $\sn$ and the same degree as in~$T$.

\begin{corollary}\label{cor:pieces}
$\gv(T)=\bigoplus_{Q}\gv(Q)$ and $\m(T)=\sum_Q |E(Q)|$, where $Q$ runs over the pieces of $T$.
\end{corollary}

\section{Runs: nim heaps inside the Game of Arrows}\label{sec:runs}

The basic components are paths, and the computation in this section is the key to almost everything that follows.

\begin{definition}\label{def:run}
Let $k\ge 0$ and $\alpha,\beta\in\{\free,\rgt,\lft\}$. The \emph{run} $\R_k(\alpha,\beta)$ is the arrow game on the status-labelled path $v_0v_1\cdots v_k$, with edges $e_j=v_{j-1}v_j$, in which the interior vertices have status $\sn$. The ends are determined by $\alpha$ and $\beta$:
\begin{itemize}[leftmargin=*]
\item $\alpha=\free$ means $s(v_0)=\sff$;
\item $\alpha={\rgt}$ means $s(v_0)=\si$, so $e_1$ may only be oriented $v_0\to v_1$;
\item $\alpha={\lft}$ means $s(v_0)=\so$;
\item symmetrically, $\beta=\free,\rgt,\lft$ means $s(v_k)=\sff,\so,\si$.
\end{itemize}
In words, $\free$ is a \emph{free end}, and an arrow is an end \emph{blocked} in the indicated direction. The run $\R_0$ is the empty game.
\end{definition}

\begin{lemma}\label{lem:runmoves}
In $\R_k(\alpha,\beta)$, orienting $e_j$ in direction $\delta\in\{\rgt,\lft\}$ is legal if and only if both of the following hold: $j\ne1$ or $\alpha\in\{\free,\delta\}$; and $j\ne k$ or $\beta\in\{\free,\delta\}$. The resulting game is $\R_{j-1}(\alpha,\delta)+\R_{k-j}(\delta,\beta)$.
\end{lemma}

\begin{proof}
Interior vertices have $d=2$, so by \cref{lem:legal} only the ends impose conditions, and they are exactly the stated ones. Say $\delta={\rgt}$. Then $v_{j-1}$ gains an outgoing arc. If $j\ge2$, it becomes a vertex of status $\so$ with $d=1$, which is a right end of type $\rgt$ for the path $v_0\cdots v_{j-1}$. Likewise $v_j$ becomes a left end of type $\rgt$ for $v_j\cdots v_k$. If $j=1$ (resp.\ $j=k$), the end vertex becomes isolated and disappears. The case $\delta={\lft}$ is symmetric.
\end{proof}

Let $\bar\alpha$ denote $\alpha$ with the arrow reversed, where $\bar\free=\free$. Reflecting the path maps $\R_k(\alpha,\beta)$ to $\R_k(\bar\beta,\bar\alpha)$, and reversing all arrows maps $\R_k(\alpha,\beta)$ to $\R_k(\bar\alpha,\bar\beta)$. Both are isomorphisms of games. Composing them, $\R_k(\alpha,\beta)\cong\R_k(\beta,\alpha)\cong\R_k(\bar\alpha,\bar\beta)$. Hence there are only four Grundy sequences:
\[
  A_k=\gv\R_k(\free,\free),\quad B_k=\gv\R_k(\free,\rgt),\quad C_k=\gv\R_k(\rgt,\rgt),\quad D_k=\gv\R_k(\rgt,\lft).
\]

\begin{theorem}[Runs]\label{thm:runs}
For every $k\ge0$,
\[
  A_k=C_k=k \bmod 2,\qquad B_k=k,\qquad D_k=\begin{cases}1,& k\ge2 \text{ even},\\ 0,&\text{otherwise.}\end{cases}
\]
In particular, a run with one free end and one blocked end is equivalent to a nim heap whose size is its length.
\end{theorem}

\begin{proof}
We use strong induction on $k$. For $k=0$ all four values are $0$. Let $k\ge1$ and assume the formulas for all smaller indices. By the induction hypothesis, for $1\le h<k$,
\begin{equation}\label{eq:CD}
  C_h=h\bmod2,\qquad D_h=1-(h\bmod 2),\qquad\text{so}\qquad\{C_h,D_h\}=\{0,1\}.
\end{equation}
For $h=1$ this reads $C_1=1$ and $D_1=0$. We use repeatedly that $a\oplus b\equiv a+b\pmod 2$. All options below come from \cref{lem:runmoves}.

\smallskip\noindent\emph{$A_k$.} Every move gives $B_{j-1}\oplus B_{k-j}=(j-1)\oplus(k-j)$, which is congruent to $k-1$ modulo $2$. If $k$ is even, all options are odd, so $A_k=0$. If $k$ is odd, then $j=(k+1)/2$ gives the option $0$, and no option equals $1$, so $A_k=1$.

\smallskip\noindent\emph{$B_k$.} For $1\le j\le k-1$, the two directions give $B_{j-1}\oplus C_{k-j}$ and $B_{j-1}\oplus D_{k-j}$. By \eqref{eq:CD} these are $j-1$ and $(j-1)\oplus1$. For $j=k$ only the direction $\rgt$ is legal, giving $B_{k-1}=k-1$. So the set of options is
\[
  \{\,j-1,\ (j-1)\oplus 1 : 1\le j\le k-1\,\}\cup\{k-1\}=\{0,1,\dots,k-1\}.
\]
Here every value $0,\dots,k-2$ occurs as $j-1$, and $(j-1)\oplus 1\le j\le k-1$. Hence $B_k=k$.

\smallskip\noindent\emph{$C_k$.} The edges $e_1$ and $e_k$ may only be oriented $\rgt$; interior edges may be oriented either way. The direction $\rgt$ gives $C_{j-1}\oplus C_{k-j}$, which is congruent to $(j-1)+(k-j)$ modulo $2$ (for $j\in\{1,k\}$ one term is $C_0=0$). The direction $\lft$ at $2\le j\le k-1$ gives $D_{j-1}\oplus D_{k-j}$, which by \eqref{eq:CD} is also congruent to $(j-1)+(k-j)$ modulo $2$. So every option equals $(k-1)\bmod 2$, and at least one option exists. Hence $C_k=\mex\{(k-1)\bmod 2\}=k\bmod 2$.

\smallskip\noindent\emph{$D_k$.} If $k=1$, the only edge would have to be oriented both ways, so there is no move and $D_1=0$. Let $k\ge2$. Orienting $e_1$ (necessarily $\rgt$) or $e_k$ (necessarily $\lft$) gives $D_{k-1}$. For $2\le j\le k-1$, the options are $C_{j-1}\oplus D_{k-j}$ and $D_{j-1}\oplus C_{k-j}$, and by \eqref{eq:CD} both equal $1\oplus((k-1)\bmod 2)=k\bmod 2$. If $k=2$, the only option is $D_1=0$, so $D_2=1$. If $k\ge3$, then $D_{k-1}=1-((k-1)\bmod 2)=k\bmod 2$ as well, so all options equal $k\bmod2$ and $D_k=1-(k\bmod 2)$.
\end{proof}

\begin{corollary}[Paths]\label{cor:path}
A path with $n\ge3$ vertices has $\gv=(n-3)\bmod 2=\m\bmod2$. If $n\ge4$ and one markable edge has been marked, the remaining position has Grundy value of the same parity as $\m-1=n-4$. If the marked edge is one of the two end markable edges, the value is exactly $n-4$.
\end{corollary}

\begin{proof}
The path has a single piece $\R_{n-3}(\free,\free)$. Marking its $j$-th edge leaves $\R_{j-1}(\free,\delta)+\R_{n-3-j}(\delta,\free)$, whose value is $(j-1)\oplus(n-3-j)\equiv n-4\pmod2$, and which equals $n-4$ when $j\in\{1,n-3\}$.
\end{proof}

\Cref{cor:path} recovers \cite[Theorem~5.1]{Lin2021} and \cite[Theorem~2.7, Corollary~2.8]{Zandee2022}. Zandee's Theorem~2.7 is precisely the statement $B_k=k$ for runs arising from a path. The point of \cref{thm:runs} is that every position on a tree decomposes into runs as soon as its branch vertices are resolved, and runs of type $B$ are nim heaps of unbounded size.

\section{Splitting, gluing and twins}\label{sec:consequences}

The run formulas have pleasant structural consequences. Vertices next to leaves act as hinges, where trees can be taken apart and glued together, and two identical branches at a vertex cancel each other out.

\subsection{Gluing}
Let $T_1,T_2$ be trees with special vertices $w_1\in T_1$ and $w_2\in T_2$. Let $T_1\circ T_2$ denote the tree obtained from the disjoint union by identifying $w_1$ with $w_2$. (The choice of $w_1,w_2$ is suppressed from the notation.)

\begin{theorem}[Gluing]\label{thm:glue}
$\gv(T_1\circ T_2)=\gv(T_1)\oplus\gv(T_2)$ and $\m(T_1\circ T_2)=\m(T_1)+\m(T_2)$.
\end{theorem}

\begin{proof}
The merged vertex $w$ has degree at least $4$ and is adjacent to a leaf, so it is special. Every other vertex keeps its neighbourhood. Therefore the leaves, the pendant edges and the special vertices of $T_1\circ T_2$ are exactly those of $T_1$ and $T_2$, with $w_1$ and $w_2$ merged. Deleting the pendant edges and cutting at the special vertices yields the pieces of $T_1$ together with the pieces of $T_2$. Now apply \cref{cor:pieces}.
\end{proof}

\begin{remark}\label{rem:plusone}
A cheaper way to add $1$: attach a path $w\,a\,b$ with two edges at a special vertex $w$. This makes $a$ special, so the new markable edge $wa$ forms a piece $\R_1(\free,\free)$ by itself. The operation increases $\m$ by $1$ and nim-adds $1$ to the Grundy value, using only two new vertices.
\end{remark}

Every tree with at least three vertices has a special vertex, and the trees with at most two vertices have Grundy value $0$. Hence:

\begin{corollary}\label{cor:subgroup}
The set $\G=\{\gv(T): T\text{ a finite tree}\}$ is closed under $\oplus$, i.e.\ it is an $\mathbb{F}_2$-subspace of $(\NN,\oplus)$. In particular, if $2^0,2^1,\dots,2^{r-1}\in\G$, then $\{0,1,\dots,2^r-1\}\subseteq\G$.
\end{corollary}

\begin{proposition}[Reformulation]\label{prop:reform}
The following are equivalent.
\begin{enumerate}[label=\textup{(\roman*)}]
\item The parity conjecture holds for all trees.
\item $\gv(T)=\m(T)\bmod 2$ for every tree $T$.
\end{enumerate}
In particular, the parity conjecture implies $\gv(T)\le1$ for all trees $T$, so any single tree with Grundy value at least $2$ refutes it.
\end{proposition}

\begin{proof}
(ii)$\Rightarrow$(i) is clear. Assume (i), and suppose some tree $T$ has $\gv(T)\ne \m(T)\bmod2$. If $\m(T)$ were even, (i) would give $\gv(T)=0$. So $\m(T)$ is odd and $\gv(T)\notin\{0,1\}$, i.e.\ $\gv(T)\ge2$. Let $P_4$ be the path on four vertices, with $\m(P_4)=1$ and $\gv(P_4)=1$. By \cref{thm:glue}, $T\circ P_4$ has an even number of markable edges and Grundy value $\gv(T)\oplus1\ne0$, which contradicts~(i).
\end{proof}

Note that a tree with $\gv(T)\ge2$ violates the conjecture itself if $\m(T)$ is even. If $\m(T)$ is odd, then $T\circ P_4$ violates it.

\begin{corollary}\label{cor:leafbranch}
Suppose every vertex of degree at least $3$ in $T$ is adjacent to a leaf; for instance, $T$ is a path, a caterpillar, or a spider whose centre is adjacent to a leaf. Then $\gv(T)=\m(T)\bmod 2$.
\end{corollary}

\begin{proof}
Every vertex of degree at least $3$ is special, so all pieces are paths with free ends, i.e.\ runs $\R_k(\free,\free)$. By \cref{cor:pieces,thm:runs}, $\gv(T)=\bigoplus_Q (k_Q\bmod2)=\m(T)\bmod 2$.
\end{proof}

\subsection{Twin branches}
Let $c$ be a vertex of $T$. For a component $B$ of $T-c$, let $T[B+c]$ denote the subtree induced by $B\cup\{c\}$. Two distinct components $B_1,B_2$ of $T-c$ are \emph{twins at $c$} if there is a graph isomorphism $\varphi\colon T[B_1+c]\to T[B_2+c]$ with $\varphi(c)=c$.

\begin{theorem}[Twin branches]\label{thm:twin}
Let $B_1,B_2$ be twins at $c$. Let $T'$ be obtained from $T$ by deleting $B_1\cup B_2$ and attaching a new pendant edge at $c$. Then $\gv(T)=\gv(T')$ and $\m(T)\equiv\m(T')\pmod 2$.
\end{theorem}

\begin{proof}
We show that $T+T'$ is a second-player win. Let $R$ be the vertex set of $T$ outside $B_1\cup B_2$. The tree $T[R]$ appears both in $T$ and in $T'$; let $\psi$ map each edge of $T[R]\subseteq T$ to the corresponding edge of $T'$. Extend $\varphi$ to an involution of $B_1\cup B_2\cup\{c\}$ by using $\varphi^{-1}$ on $B_2$. After each of her moves, the second player maintains the following invariant.
\begin{enumerate}[label=(\arabic*)]
\item For each edge $e$ of $T[B_1+c]$, the edges $e$ and $\varphi(e)$ are both unmarked, or both marked with \emph{opposite} orientations: $x\to y$ on $e$ if and only if $\varphi(y)\to\varphi(x)$ on $\varphi(e)$.
\item Every edge $e$ of $T[R]\subseteq T$ is in the same state as $\psi(e)$ in $T'$.
\end{enumerate}
The invariant holds at the start. When the first player marks an edge $e$, the second player answers as follows.
\begin{itemize}[leftmargin=*]
\item If $e$ lies in $T[B_1+c]$ or $T[B_2+c]$, she marks $\varphi(e)$ with the opposite orientation.
\item If $e$ lies in $T[R]\subseteq T$, she marks $\psi(e)$ with the same orientation.
\item If $e$ lies in $T'$, she marks $\psi^{-1}(e)$ with the same orientation. (The new pendant edge of $T'$ is never marked.)
\end{itemize}
By the invariant the answering edge $e^*$ is unmarked, and afterwards the invariant holds again. It remains to show that the answer is legal, i.e.\ that neither endpoint of $e^*$ is a sink or a source afterwards.

Let $w^*$ be an endpoint of $e^*$ with $w^*\ne c$, and let $w$ be the corresponding endpoint of $e$. The answering map ($\varphi$ or $\psi^{\pm1}$) restricts to a bijection from the edges at $w$ to the edges at $w^*$. The edge $e^*$ is not incident with $w$, so after both moves the edges at $w^*$ are in the states, or in the reversed states in the twin case, of the edges at $w$ just after the first player's legal move. So $w^*$ is neither a sink nor a source.

Now suppose $c$ is an endpoint of $e^*$.
\begin{itemize}[leftmargin=*]
\item If $e^*$ lies in $T'$, then $e^*\in\psi(T[R])$, so $R\ne\{c\}$. Hence $c$ has degree at least~$2$ in $T'$ and is adjacent to the new leaf, i.e.\ it is special. By \cref{lem:pendant} it is never a sink or source.
\item If $e^*$ lies in $T$, let $b_i$ be the neighbour of $c$ in $B_i$, so $\varphi(cb_1)=cb_2$. After the answer, by the invariant, the edges $cb_1$ and $cb_2$ are either both unmarked, in which case $c$ keeps an unmarked edge, or marked with opposite orientations relative to $c$, in which case $c$ has an incoming and an outgoing arc.
\end{itemize}
In every case the answer is legal. So the second player always has a reply, and since the game is finite she wins. Hence $\gv(T)\oplus\gv(T')=0$.

For the parity statement, note that $c$ is not a leaf of $T$, and the vertices of $B_i$ have the same degree in $T$ as in $T[B_i+c]$. So $\varphi$ maps the edges of $T[B_1+c]$ that are markable in $T$ bijectively onto those of $T[B_2+c]$. The new pendant edge of $T'$ is not markable. An edge of $T[R]$ at $c$ is markable in $T$ if and only if its image is markable in $T'$; if $R=\{c\}$ there are no such edges. Thus
\[
\m(T)=\m(T')+2\,\#\{\text{edges of } T[B_1+c] \text{ markable in } T\}.\qedhere
\]
\end{proof}

\begin{corollary}\label{cor:twinspider}
For all $x,z\ge1$, $\gv(\Sp(x,x,z))=z\bmod 2$. More generally,
\[\gv(\Sp(x,x,y_1,\dots,y_r))=(y_1+\dots+y_r)\bmod 2.\]
\end{corollary}

\begin{proof}
Replace the twin legs by a pendant edge. The centre becomes special, and the rest splits into the runs $\R_{y_t}(\free,\free)$.
\end{proof}

The first statement is a Grundy-value form of Lin's theorem on the spiders $T_{a,a,b}$~\cite[Theorem~5.3]{Lin2021}. \Cref{thm:twin} is also a tree analogue of the involution strategy of Alvarado et al.~\cite{Alvarado2021}. It shows that, as far as the parity conjecture is concerned, we may restrict attention to trees without twin branches.

\section{Spiders}\label{sec:spiders}

Spiders are where the action is. By \cref{cor:leafbranch,thm:twin}, the simplest pieces not covered so far are spiders with pairwise distinct legs and no leaf at the centre. We now describe all their positions. Fix the spider $\Sp(x_1,\dots,x_k)$ with centre $c$.

\subsection{Spider states}
Directions on a leg are measured relative to the centre: $\IN$ means towards $c$, and $\OUT$ means away from $c$. For $s\in\{\sn,\si,\so,\sff\}$ we write $s\cdot\IN=s\cdot\si$ and $s\cdot\OUT=s\cdot\so$; this is the status of $c$ after an arc at $c$ pointing into, respectively out of, $c$ (\cref{lem:legal}).

In any position, look along leg $t$ from the centre. Let $h_t$ be the number of unmarked edges before the first marked edge of the leg; if no edge of the leg is marked, count the unmarked edges before the pendant edge. The leg is \emph{attached} if $h_t\ge1$, i.e.\ if its edge at $c$ is unmarked. The \emph{far end} $\phi_t$ of the attached part is $\free$ if the leg contains no arc. Otherwise it is the direction ($\IN$ or $\OUT$) of the first arc, which by \cref{lem:runmoves} forces the direction of the adjacent unmarked edge. Everything beyond the first arc on the leg splits into runs, whose values are given by \cref{thm:runs}.

A \emph{spider state} is a pair $(s;L)$, where $s\in\{\sn,\si,\so,\sff\}$ is the status of $c$ and $L$ is the multiset of pairs $(h_t,\phi_t)$ over the attached legs. While $s=\sn$, all $k$ legs are attached.

For $h\ge0$, $\delta\in\{\IN,\OUT\}$ and $\phi\in\{\free,\IN,\OUT\}$, let $\rho(h;\delta,\phi)$ be the value of a detached run of $h$ edges whose near end is blocked in direction $\delta$ and whose far end is $\phi$. By \cref{thm:runs},
\[
\rho(h;\delta,\free)=B_h=h,\qquad \rho(h;\delta,\phi)=C_h \text{ if } \phi=\delta,\qquad \rho(h;\delta,\phi)=D_h \text{ otherwise}.
\]
Similarly, let $\rho_\free(h;\phi)$ be the value of a run with free near end and far end $\phi$. Thus $\rho_\free(h;\free)=A_h$ and $\rho_\free(h;\IN)=\rho_\free(h;\OUT)=B_h$.

Let $\Gamma(s;L)$ be the Grundy value of the component containing $c$. In the following degenerate cases it is given directly.
\begin{enumerate}[label=(N\arabic*),leftmargin=*]
\item If $s=\sff$, then $c$ splits and every attached leg becomes a run with free near end, so $\Gamma(\sff;L)=\bigoplus_{(h,\phi)\in L}\rho_\free(h;\phi)$.
\item If $s\in\{\si,\so\}$ and $L=\{(h,\phi)\}$ has one element, the remaining edge at $c$ is forced ($\OUT$ if $s=\si$, $\IN$ if $s=\so$). So $\Gamma(\si;L)=\rho(h;\OUT,\phi)$ and $\Gamma(\so;L)=\rho(h;\IN,\phi)$.
\item If $L=\emptyset$, then $\Gamma=0$.
\end{enumerate}

\begin{proposition}\label{prop:spiderrec}
In the remaining cases ($s=\sn$, or $s\in\{\si,\so\}$ with $|L|\ge2$), $\Gamma(s;L)$ is the mex of the following option values. There is one option for each attached leg $(h,\phi)\in L$, each $1\le j\le h$ and each $\delta\in\{\IN,\OUT\}$, except that $\delta=\phi$ is required when $j=h$ and $\phi\ne\free$. Its value is
\[
  \rho(h-j;\delta,\phi)\ \oplus\ \begin{cases}
  \Gamma\bigl(s;\ L-\{(h,\phi)\}+\{(j-1,\delta)\}\bigr), & j\ge2,\\[2pt]
  \Gamma\bigl(s\cdot\delta;\ L-\{(h,\phi)\}\bigr), & j=1.
  \end{cases}
\]
The Grundy value of the spider is $\gv(\Sp(x_1,\dots,x_k))=\Gamma(\sn;\{(x_1,\free),\dots,(x_k,\free)\})$.
\end{proposition}

\begin{proof}
Consider the move on the $j$-th edge (from the centre) of the attached part of a leg, in direction $\delta$. It detaches the $h-j$ unmarked edges beyond it, which form a run blocked in direction $\delta$ at the near end and with far end $\phi$. If $j\ge2$, the part between $c$ and the new arc stays attached, with $h$ replaced by $j-1$ and far end $\delta$. If $j=1$, the arc is at $c$ and changes its status to $s\cdot\delta$. The legality conditions come from \cref{lem:legal}. At $c$ there is no condition, because in the remaining cases $c$ has $d(c)=|L|\ge2$. At the far end, the condition is $\delta=\phi$ when $j=h$ and $\phi\ne\free$. Finally, apply \cref{prop:sum}.
\end{proof}

\begin{corollary}\label{cor:poly}
For each fixed $k$, the Grundy value of a $k$-legged spider with $n$ vertices can be computed in time $O(n^{k+1})$ and space $O(n^k)$.
\end{corollary}

\begin{proof}
There are at most $4\prod_t(3x_t+1)=O(n^k)$ spider states, each with $O(n)$ options. Each option is evaluated in constant time from \cref{thm:runs} and memoised values.
\end{proof}

\begin{remark}[How large values arise]\label{rem:large}
Cutting an attached leg with free far end at its $j$-th edge creates the nim heap $\R_{h-j}(\rgt,\free)$ of value $h-j$. Hence, for $0\le t\le h-2$, the options of the state include the numbers $t\oplus\Gamma(s;\,L-\{(h,\free)\}+\{(h-t-1,\delta)\})$. If a long leg is cut near its middle, two large numbers of similar size are nim-added, and the result can be small. This is the mechanism behind the counterexamples, where a well-chosen cut produces an option of value $0$ or of some other prescribed small value. The second player's winning replies in \cref{tab:firstmoves} below are typically cuts of this kind.
\end{remark}

\section{Computations}\label{sec:computation}

Now for the experiments. The theory of the previous sections makes the search for counterexamples fast, and the computations in turn suggested most of the theorems proved in \cref{sec:s23,sec:onetwoj,sec:threads}.

\subsection{Programs and verification}\label{sec:programs}
All computations were made with the following programs. They are available as ancillary files (see the section on data and code availability at the end).
\begin{enumerate}[label=(P\arabic*),leftmargin=*]
\item \emph{Raw-rules solvers.} These are two independently written programs, one in Python and one in C++. Positions are vectors in $\{\text{unmarked},\to,\leftarrow\}^{E(T)}$, legality is checked literally (no vertex all-in or all-out), and Grundy values are memoised. No result of this paper is used.
\item \emph{Normal-form solver for arbitrary trees.} It implements \cref{prop:sum}, using canonical forms of status-labelled trees up to isomorphism and up to the global reversal $\si\leftrightarrow\so$.
\item \emph{Spider solvers} implementing \cref{prop:spiderrec}. There are four versions:
\begin{itemize}[leftmargin=*]
\item an array-based C++ solver for three legs;
\item a dictionary-based Python solver for three legs;
\item a hash-based C++ solver for any number of legs;
\item a solver for $\Sp(x,y,z)$ with $x,y$ fixed and $z$ growing, which organises the recursion around the long leg.
\end{itemize}
A fifth spider solver was written independently from the paper's text alone.
\end{enumerate}

\begin{table}[t]
\centering
\small
\begin{tabular}{p{6.4cm}p{6.6cm}}
\toprule
claim & confirmed by\\
\midrule
census, $n\le 16$ (\cref{tab:census}) & both raw-rules programs (P1), and (P2)\\
census, $17\le n\le 20$ & C++ raw-rules program (P1), and (P2)\\
\cref{thm:census}; trees in \cref{fig:cex,fig:large} & (P1) twice, (P2), and (P3) for the spiders\\
first moves of $\Sp(1,3,7)$ (\cref{tab:firstmoves}) & (P1) twice, (P2)\\
\cref{prop:zandee} & (P1) twice, (P2)\\
\cref{thm:runs,thm:glue,thm:twin} (sanity checks) & (P1) on hundreds of random instances\\
spider census $x\le y\le z\le 120$ (\cref{tab:spidervalues}) & two (P3) versions on all $295{,}240$ spiders\\
values $8,16,32$ used in \cref{thm:allvalues} & three (P3) versions\\
spider families (\cref{tab:families}) & two (P3) versions, all $92$ pairs, $z\le 600$\\
finite part of \cref{thm:table23} & \texttt{verify23.py} ($r\le 100$) and the independent \texttt{famstates.cpp} ($r\le 3000$)\\
\cref{thm:cert} & certificate program; closed forms cross-checked by an independent solver\\
\cref{thm:onetwoj} (proved by hand) & two solvers for $y\le 64$, $z\le 2000$; window of \cref{rem:onetwojwindow} for $y\le 160$\\
\cref{thm:clean} (proved by hand) & strategy checked on all positions of $\mathcal C$ with $\le 12$ vertices\\
\cref{thm:threads3} & raw-rules solver, (P2), and a third solver\\
$H$-tree values in \cref{sec:open} & two independent $H$-tree solvers\\
\bottomrule
\end{tabular}
\caption{Which programs confirmed which computational claim.}
\label{tab:verification}
\end{table}

\noindent In addition, the following cross-checks were made.
\begin{itemize}[leftmargin=*]
\item The two raw-rules programs agree tree by tree on all $13{,}187$ trees with $2$ to $15$ vertices.
\item They agree with (P2) on all $2287$ trees with $2$ to $13$ vertices.
\item (P2) and (P3) agree on all three-legged spiders with legs at most~$8$ and all four-legged spiders with legs at most~$4$.
\item The C++ and Python three-leg solvers agree on all $4960$ spiders with legs at most $30$.
\item The array-based and fixed-$x,y$ solvers agree for all $x\le6$, $y\le 8$, $z\le120$.
\item The array-based and hash-based solvers agree on a random sample of $300$ spiders with legs at most $30$ and on the spiders used in \cref{thm:allvalues}.
\end{itemize}
For the counterexample $\Sp(1,3,7)$, the ancillary files also contain a certificate that can be checked by hand. It lists the values $\Gamma$ of all $205$ spider states reachable from the empty position, together with the values of their options. Verifying it requires only \cref{thm:runs} and \cref{prop:spiderrec}.

\subsection{Census of small trees}

\begin{table}[t]
\centering
\small
\begin{tabular}{rrrrrrrrr}
\toprule
$n$ & trees & $\gv=0$ & $\gv=1$ & $\gv=2$ & $\gv=3$ & $\gv=4$ & $\gv=5$ & PC fails\\
\midrule
$2$--$11$ & 435 & 222 & 213 & 0 & 0 & 0 & 0 & 0 \\
12 & 551 & 273 & 278 & 0 & 0 & 0 & 0 & 0\\
13 & 1301 & 657 & 644 & 0 & 0 & 0 & 0 & 0\\
14 & 3159 & 1573 & 1586 & 0 & 0 & 0 & 0 & 0\\
15 & 7741 & 3887 & 3854 & 0 & 0 & 0 & 0 & 2\\
16 & 19320 & 9649 & 9671 & 0 & 0 & 0 & 0 & 7\\
17 & 48629 & 24353 & 24274 & 1 & 0 & 1 & 0 & 25\\
18 & 123867 & 61893 & 61962 & 5 & 3 & 4 & 0 & 82\\
19 & 317955 & 159059 & 158850 & 16 & 15 & 10 & 5 & 241\\
20 & 823065 & 411384 & 411522 & 56 & 51 & 28 & 24 & 726\\
\bottomrule
\end{tabular}
\caption{Grundy values of the Game of Arrows on all unlabelled trees with $n$ vertices. The last column counts the trees violating the parity conjecture (PC), i.e.\ those with $\gv=0$ and $\m$ odd, or $\gv\neq0$ and $\m$ even. This is not the same as $\gv\ne\m\bmod2$; for instance, at $n=18$ there are $84$ trees with $\gv\ne\m\bmod 2$. No Grundy value larger than $5$ occurs for $n\le 20$. The numbers of trees are OEIS A000055.}
\label{tab:census}
\end{table}

For integers $a_1,a_2,b_1,b_2\ge1$ and $\ell\ge1$, let $H(a_1,a_2\,|\,\ell\,|\,b_1,b_2)$ denote the tree with two branch vertices $u,v$ joined by a path of $\ell$ edges. The vertex $u$ carries two further legs with $a_1$ and $a_2$ markable edges, and $v$ carries two legs with $b_1$ and $b_2$ markable edges; each leg ends with a pendant edge. All edges between $u$ and $v$ are markable, so $\m=a_1+a_2+b_1+b_2+\ell$, and the tree has $a_1+a_2+b_1+b_2+\ell+5$ vertices.

\begin{theorem}\label{thm:census}
\leavevmode
\begin{enumerate}[label=\textup{(\alph*)}]
\item Every tree with at most $14$ vertices satisfies $\gv(T)=\m(T)\bmod2$.
\item Among the $7741$ trees with $15$ vertices, exactly two violate the parity conjecture: the spiders $\Sp(1,3,7)$ and $\Sp(1,4,6)$. Both have $\m=11$ and $\gv=0$.
\item Every tree with at most $16$ vertices has Grundy value $0$ or $1$. Exactly two trees with $17$ vertices have larger Grundy value, namely the trees $H(3,1\,|\,3\,|\,4,1)$ and $H(5,2\,|\,1\,|\,3,1)$ of \cref{fig:large}. Their Grundy values are $2$ and $4$, while $\m=12$ for both, so both violate the parity conjecture.
\item The complete distribution of Grundy values for $n\le 20$ is given in \cref{tab:census}.
\end{enumerate}
\end{theorem}

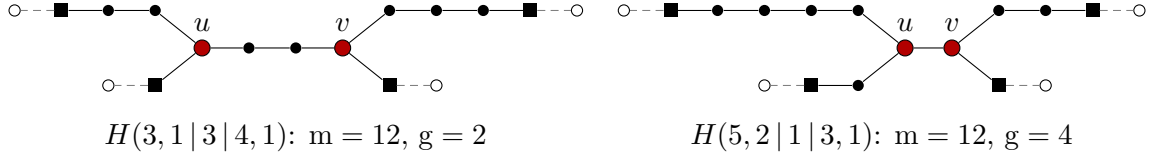
\begin{figure}[t]
\centering
\begin{tikzpicture}[scale=0.62]
\begin{scope}
\node[ctr] (u) at (0,0) {}; \node at (0,0.5) {$u$};
\node[v] (p1) at (1,0) {}; \node[v] (p2) at (2,0) {};
\node[ctr] (w) at (3,0) {}; \node at (3,0.5) {$v$};
\draw (u)--(p1)--(p2)--(w);
\node[v] (a1) at (-1,0.8) {}; \node[v] (a2) at (-2,0.8) {}; \node[sp] (a3) at (-3,0.8) {}; \node[lf] (a4) at (-4,0.8) {};
\draw (u)--(a1)--(a2)--(a3); \draw[pend] (a3)--(a4);
\node[sp] (b1) at (-1,-0.8) {}; \node[lf] (b2) at (-2,-0.8) {};
\draw (u)--(b1); \draw[pend] (b1)--(b2);
\foreach \i in {1,2,3} {\node[v] (c\i) at (3+\i,0.8) {};}
\node[sp] (c4) at (7,0.8) {}; \node[lf] (c5) at (8,0.8) {};
\draw (w)--(c1)--(c2)--(c3)--(c4); \draw[pend] (c4)--(c5);
\node[sp] (d1) at (4,-0.8) {}; \node[lf] (d2) at (5,-0.8) {};
\draw (w)--(d1); \draw[pend] (d1)--(d2);
\node at (2,-1.9) {$H(3,1\,|\,3\,|\,4,1)$: $\m=12$, $\gv=2$};
\end{scope}
\begin{scope}[xshift=15cm]
\node[ctr] (u) at (0,0) {}; \node at (0,0.5) {$u$};
\node[ctr] (w) at (1,0) {}; \node at (1,0.5) {$v$};
\draw (u)--(w);
\foreach \i in {1,...,4} {\node[v] (a\i) at (-\i,0.8) {};}
\node[sp] (a5) at (-5,0.8) {}; \node[lf] (a6) at (-6,0.8) {};
\draw (u)--(a1); \foreach \i [evaluate=\i as \j using int(\i+1)] in {1,...,4} {\draw (a\i)--(a\j);}
\draw[pend] (a5)--(a6);
\node[v] (b1) at (-1,-0.8) {}; \node[sp] (b2) at (-2,-0.8) {}; \node[lf] (b3) at (-3,-0.8) {};
\draw (u)--(b1)--(b2); \draw[pend] (b2)--(b3);
\node[v] (c1) at (2,0.8) {}; \node[v] (c2) at (3,0.8) {}; \node[sp] (c3) at (4,0.8) {}; \node[lf] (c4) at (5,0.8) {};
\draw (w)--(c1)--(c2)--(c3); \draw[pend] (c3)--(c4);
\node[sp] (d1) at (2,-0.8) {}; \node[lf] (d2) at (3,-0.8) {};
\draw (w)--(d1); \draw[pend] (d1)--(d2);
\node at (-0.5,-1.9) {$H(5,2\,|\,1\,|\,3,1)$: $\m=12$, $\gv=4$};
\end{scope}
\end{tikzpicture}
\caption{The two trees with $17$ vertices whose Grundy value exceeds $1$. Both are counterexamples to the parity conjecture: $\m=12$ is even, but the Grundy value is non-zero.}
\label{fig:large}
\end{figure}

\Cref{tab:firstmoves} lists every first move on $\Sp(1,3,7)$ up to the global reversal of arrows. For each it gives the Grundy value of the resulting position and the winning replies of the second player. The legs are named $X,Y,Z$, with $1,3,7$ markable edges, as in \cref{fig:cex}. Edges are numbered from the centre, and $\IN/\OUT$ are as in \cref{sec:spiders}.

Every first move leads to a position with even, \emph{non-zero} Grundy value. These values can be traced back to \cref{thm:runs}. For instance, after the first player plays $Z_4\OUT$, the cut leaves the nim heap $\R_3(\rgt,\free)$ of value $3$. The reply $Z_2\IN$ then leaves a dead edge and a centre component of value $3$, and the total nim-sum is $0$.

\begin{table}[t]
\centering
\small
\begin{tabular}{lcl@{\qquad}lcl}
\toprule
first move & value & winning replies & first move & value & winning replies\\
\midrule
$X_1\OUT$ & 2 & $Z_3\IN$ & $Z_2\OUT$ & 6 & $Z_4\IN$, $Z_5\IN$\\
$Y_1\OUT$ & 2 & $Y_3\IN$ & $Z_3\OUT$ & 6 & $X_1\IN$, $Z_5\IN$\\
$Y_2\OUT$ & 6 & $Z_6\OUT$, $Z_7\OUT$ & $Z_4\OUT$ & 4 & $Z_2\IN$\\
$Y_3\OUT$ & 6 & $Y_1\IN$, $Z_6\OUT$, $Z_7\OUT$ & $Z_5\OUT$ & 6 & $Z_2\IN$, $Z_3\IN$\\
$Z_1\OUT$ & 6 & $Z_6\IN$, $Z_7\IN$ & $Z_6\OUT$ & 4 & $Y_2\OUT$, $Y_3\OUT$, $Z_1\IN$\\
 & & & $Z_7\OUT$ & 8 & $Y_2\OUT$, $Y_3\OUT$, $Z_1\IN$\\
\bottomrule
\end{tabular}
\caption{All first moves on $\Sp(1,3,7)$ up to reversal of all arrows; the move $E\,\IN$ has the same value as $E\,\OUT$, with all replies reversed. Since $0$ never occurs as the value of a first move, $\gv(\Sp(1,3,7))=0$ although $\m=11$ is odd.}
\label{tab:firstmoves}
\end{table}

\subsection{Zandee's conjecture on one marked edge}
Zandee~\cite[Conjecture~4.3]{Zandee2022} conjectured the following. If a tree $T$ has exactly one marked edge, then the Grundy value of the resulting position has the same parity as $\m(T)-1$. For paths this follows from \cref{thm:runs} (see \cref{cor:path}).

\begin{proposition}\label{prop:zandee}
Zandee's conjecture fails. In $\Sp(1,3,4)$, which has $12$ vertices and $\m=8$, mark the outermost markable edge of the leg with $3$ markable edges, pointing away from the centre. The resulting position has Grundy value $4$, but $\m-1=7$ is odd. The trees $\Sp(1,3,4)$ and $\Sp(2,3,3)$ are the only counterexamples with at most $12$ vertices, and there are none with at most $11$ vertices.
\end{proposition}

In $\Sp(2,3,3)$, marking the outermost markable edge of a leg with three markable edges gives the value $6$. The parity conjecture itself holds for both trees: $\gv(\Sp(1,3,4))=\gv(\Sp(2,3,3))=0$ and $\m=8$. Note also that a counterexample to one of the two conjectures need not be a counterexample to the other. For instance, all first moves on the counterexample $\Sp(1,3,7)$ lead to even values (\cref{tab:firstmoves}), as Zandee's conjecture predicts. By contrast, the other $15$-vertex counterexample $\Sp(1,4,6)$ has a first move with the odd value $7$, namely cutting its longest leg at the third edge, although $\m-1=10$.

\subsection{Large Grundy values}

\begin{table}[t]
\centering
\small
\begin{tabular}{rl@{\qquad}rl@{\qquad}rl}
\toprule
$g$ & spider & $g$ & spider & $g$ & spider\\
\midrule
1 & $\Sp(1,1,1)$ & 11 & $\Sp(11,26,29)$ & 21 & $\Sp(15,84,85)^{\dagger}$\\
2 & $\Sp(3,4,7)$ & 12 & $\Sp(7,23,28)$ & 22 & $\Sp(25,48,59)^{\dagger}$\\
3 & $\Sp(2,3,9)$ & 13 & $\Sp(13,24,29)$ & 23 & $\Sp(21,58,77)^{\dagger}$\\
4 & $\Sp(5,8,9)$ & 14 & $\Sp(7,29,34)$ & 24 & $\Sp(12,25,31)$\\
5 & $\Sp(2,5,21)$ & 15 & $\Sp(21,24,39)$ & 26 & $\Sp(21,53,64)^{\dagger}$\\
6 & $\Sp(5,9,10)$ & 16 & $\Sp(13,44,51)$ & 27 & $\Sp(17,86,87)^{\dagger}$\\
7 & $\Sp(9,14,16)$ & 17 & $\Sp(15,27,36)$ & 28 & $\Sp(23,49,62)^{\dagger}$\\
8 & $\Sp(7,12,19)$ & 18 & $\Sp(13,45,48)$ & 30 & $\Sp(39,101,116)^{\dagger}$\\
9 & $\Sp(9,14,17)$ & 19 & $\Sp(15,34,37)$ & 32 & $\Sp(17,90,95)^{\dagger}$\\
10 & $\Sp(7,13,18)$ & 20 & $\Sp(40,45,87)^{\dagger}$ & & \\
\bottomrule
\end{tabular}
\caption{For each value $g$, the three-legged spider with the fewest markable edges ($\m=x+y+z$) having Grundy value $g$, among all spiders with legs of at most $120$ markable edges; ties are broken lexicographically. For entries without $\dagger$ this is the minimum over \emph{all} three-legged spiders, because $\m\le 122$ forces every leg to be at most $120$. The values $25,29,31$ and all values above $32$ do not occur in this range.}
\label{tab:spidervalues}
\end{table}

We computed $\gv(\Sp(x,y,z))$ for all $1\le x\le y\le z\le 120$, a total of $295{,}240$ spiders. The values $0,\dots,24$, $26,27,28,30,32$ occur, and \cref{tab:spidervalues} lists the smallest example of each.

\begin{theorem}\label{thm:allvalues}
Every integer $0\le g\le 63$ is the Grundy value of the Game of Arrows on some tree. In particular, there are trees with Grundy value $4$ or larger, which answers the question of Fokkink and Zandee~\cite{Fokkink2024}. The smallest tree with Grundy value at least $2$ has $17$ vertices, and the same holds for Grundy value $4$.
\end{theorem}

\begin{proof}
By \cref{tab:spidervalues}, the values $1,2,4,8,16,32$ are attained by the spiders $\Sp(1,1,1)$, $\Sp(3,4,7)$, $\Sp(5,8,9)$, $\Sp(7,12,19)$, $\Sp(13,44,51)$ and $\Sp(17,90,95)$. Write $g$ in binary and glue the corresponding spiders at special vertices. By \cref{thm:glue} the resulting tree has Grundy value $g$. The last statement is \cref{thm:census}(c).
\end{proof}

\Cref{tab:census} also shows that the value $3$ first occurs only at $18$ vertices, later than the value $4$; a smallest example is $\Sp(2,3,9)$. The value $5$ first occurs at $19$ vertices. An example is obtained from $H(5,2\,|\,1\,|\,3,1)$ by attaching a path with two edges at a special vertex (\cref{rem:plusone}). Its pieces are the piece of $H(5,2\,|\,1\,|\,3,1)$, of value $4$, and a single edge, of value $1$.

\subsection{Families of spiders}\label{sec:families}
For fixed $x\le y$ we computed $z\mapsto\gv(\Sp(x,y,z))$ for $1\le z\le 600$. We did this for all $78$ pairs with $1\le x\le y\le 12$ and for the $14$ pairs with $x\le 7$ and $y\in\{13,14\}$, $92$ pairs in total. \Cref{tab:families} summarises the results for $y\le 12$. Three behaviours occur.

\begin{enumerate}[label=(\roman*),leftmargin=*]
\item If $x=y$, then \cref{cor:twinspider} gives $\gv=z\bmod 2$ for all $z$.
\item In all other computed pairs except those with $x=1$ and $y$ even, there is a threshold $z_0(x,y)$ such that
\[
  \gv(\Sp(x,y,z))=z\bmod 2\qquad\text{for all } z_0\le z\le 600.
\]
Since $\m=x+y+z$, the parity rule then holds for all such $z$ when $x\equiv y\pmod 2$, and fails for all such $z$ when $x\not\equiv y\pmod2$. In the latter case the spider has Grundy value $1$ with an even number of markable edges, or Grundy value $0$ with an odd number. For example, $\Sp(2,3,z)$ violates the parity conjecture for every $14\le z\le 600$. The thresholds are irregular; for instance $z_0(1,3)=z_0(2,3)=14$, $z_0(2,7)=34$, $z_0(1,9)=z_0(2,9)=56$, and the largest observed threshold is $z_0(4,13)=70$.
\item If $x=1$ and $y=2$, then $\gv=(z+1)\bmod 2$ for all computed $z$, so the parity rule holds. If $x=1$ and $y\ge4$ is even, the sequence is periodic from early on, with period $8$ for $y=4$, $16$ for $y\in\{6,8\}$, and $32$ for $y\in\{10,12,14\}$. Within each period some values follow the parity rule and some violate it. For example, for $z\ge3$ the values $\gv(\Sp(1,4,z))$ repeat the pattern $0,1,0,0,1,0,1,0$, and the parity rule fails exactly when $z\equiv 0,1,2,6,7\pmod 8$.
\end{enumerate}

These observations suggest the following conjecture.

\begin{conjecture}[Eventual parity for spiders]\label{conj:eventual}
Let $1\le x\le y$ with $(x,y)$ not of the form $(1,2j)$. Then $\gv(\Sp(x,y,z))=z\bmod2$ for all sufficiently large $z$. Consequently, for all sufficiently large $z$ the spider $\Sp(x,y,z)$ satisfies the parity conjecture if $x\equiv y\pmod2$ and violates it if $x\not\equiv y\pmod 2$.
\end{conjecture}

In \cref{sec:s23} we prove \cref{conj:eventual} for $(x,y)=(2,3)$ by hand, apart from a finite computation, and, by a computer-assisted argument, for all $73$ computed pairs and for $(2,y)$ with $y\le 40$. In \cref{sec:onetwoj} we determine $\gv(\Sp(1,2j,z))$ exactly, which explains behaviour~(iii).

\begin{table}[t]
\centering
\footnotesize
\setlength{\tabcolsep}{3.2pt}
\begin{tabular}{c|ccccccccccc}
\toprule
$x\backslash y$ & 2 & 3 & 4 & 5 & 6 & 7 & 8 & 9 & 10 & 11 & 12\\
\midrule
1 & $\mathrm{P}^{*}\,1$ & $\mathrm{P}\,14$ & $p{=}8$ & $\mathrm{P}\,28$ & $p{=}16$ & $\mathrm{P}\,34$ & $p{=}16$ & $\mathrm{P}\,56$ & $p{=}32$ & $\mathrm{P}\,22$ & $p{=}32$\\
2 & $\mathrm{P}\,1$ & $\mathrm{A}\,14$ & $\mathrm{P}\,1$ & $\mathrm{A}\,28$ & $\mathrm{P}\,1$ & $\mathrm{A}\,34$ & $\mathrm{P}\,4$ & $\mathrm{A}\,56$ & $\mathrm{P}\,4$ & $\mathrm{A}\,22$ & $\mathrm{P}\,8$\\
3 &  & $\mathrm{P}\,1$ & $\mathrm{A}\,15$ & $\mathrm{P}\,26$ & $\mathrm{A}\,13$ & $\mathrm{P}\,32$ & $\mathrm{A}\,15$ & $\mathrm{P}\,54$ & $\mathrm{A}\,25$ & $\mathrm{P}\,10$ & $\mathrm{A}\,23$\\
4 &  &  & $\mathrm{P}\,1$ & $\mathrm{A}\,30$ & $\mathrm{P}\,16$ & $\mathrm{A}\,34$ & $\mathrm{P}\,16$ & $\mathrm{A}\,58$ & $\mathrm{P}\,24$ & $\mathrm{A}\,25$ & $\mathrm{P}\,24$\\
5 &  &  &  & $\mathrm{P}\,1$ & $\mathrm{A}\,26$ & $\mathrm{P}\,34$ & $\mathrm{A}\,29$ & $\mathrm{P}\,50$ & $\mathrm{A}\,27$ & $\mathrm{P}\,21$ & $\mathrm{A}\,28$\\
6 &  &  &  &  & $\mathrm{P}\,1$ & $\mathrm{A}\,30$ & $\mathrm{P}\,16$ & $\mathrm{A}\,54$ & $\mathrm{P}\,28$ & $\mathrm{A}\,29$ & $\mathrm{P}\,30$\\
7 &  &  &  &  &  & $\mathrm{P}\,1$ & $\mathrm{A}\,31$ & $\mathrm{P}\,32$ & $\mathrm{A}\,29$ & $\mathrm{P}\,21$ & $\mathrm{A}\,28$\\
8 &  &  &  &  &  &  & $\mathrm{P}\,1$ & $\mathrm{A}\,52$ & $\mathrm{P}\,32$ & $\mathrm{A}\,26$ & $\mathrm{P}\,34$\\
9 &  &  &  &  &  &  &  & $\mathrm{P}\,1$ & $\mathrm{A}\,54$ & $\mathrm{P}\,54$ & $\mathrm{A}\,53$\\
10 &  &  &  &  &  &  &  &  & $\mathrm{P}\,1$ & $\mathrm{A}\,33$ & $\mathrm{P}\,32$\\
11 &  &  &  &  &  &  &  &  &  & $\mathrm{P}\,1$ & $\mathrm{A}\,40$\\
\bottomrule
\end{tabular}
\caption{Eventual behaviour of $z\mapsto\gv(\Sp(x,y,z))$ for $z\le 600$. An entry $\mathrm{P}\,z_0$ means that $\gv=z\bmod 2$ for $z_0\le z\le 600$, and that this agrees with the parity rule ($x\equiv y\pmod 2$). An entry $\mathrm{A}\,z_0$ means that $\gv=z\bmod 2$ for $z_0\le z\le 600$, and that this violates the parity rule for every such $z$ ($x\not\equiv y\pmod 2$). In both cases $z_0$ is the least such value. The entry $\mathrm{P}^{*}\,1$ for $(x,y)=(1,2)$ means $\gv=(z+1)\bmod 2$ for all $z$, which again agrees with the parity rule. An entry $p{=}k$ means the sequence is periodic with period $k$ from small $z$, mixing both rules. The pair $(12,12)$, not shown, gives $\mathrm{P}\,1$.}
\label{tab:families}
\end{table}

\section{Eventual parity for spiders with two fixed legs}\label{sec:s23}

Tables are persuasive, but \cref{tab:families} stops at $z=600$. In this section we cross that finite horizon and prove \cref{conj:eventual} for the pair $(x,y)=(2,3)$. This gives an infinite, twin-free family of counterexamples to the parity conjecture. The proof is by hand, except that $52$ explicit integer sequences are checked by computer on a finite range, and one small constant is computed.

In \cref{sec:cert} we isolate the finitely many facts the argument needs, so that a program can check them for other pairs. This proves \cref{conj:eventual} for all $73$ pairs computed in \cref{sec:families} that are neither twins nor of the form $(1,\text{even})$, and also for $(2,y)$ with $y\le 40$ (\cref{thm:cert}). The thresholds observed in \cref{tab:families} for $z\le 600$ therefore hold for all~$z$. The excluded family $(1,2j)$ is solved completely, by hand, in \cref{sec:onetwoj}.

\begin{theorem}\label{thm:s23}
For every $z\ge 14$,
\[
  \gv(\Sp(2,3,z))=z\bmod 2 .
\]
Consequently, for every $z\ge 14$ the spider $\Sp(2,3,z)$ violates the parity conjecture. It has $\m=z+5$ markable edges, and its Grundy value is $0$ when $\m$ is odd and $1$ when $\m$ is even. The bound is sharp: $\gv(\Sp(2,3,13))=0$.
\end{theorem}

\subsection{Long-leg families}
The plan is to record the state of the two short legs and to watch how the value changes as the long leg grows. A \emph{configuration} is a pair $\kappa=(s;K)$, where $s\in\{\sn,\si,\so\}$ and $K$ is a multiset of pairs $(h,\psi)$ with $h\ge1$ and $\psi\in\{\free,\IN,\OUT\}$. We require $|K|=2$ if $s=\sn$ and $|K|=1$ if $s\in\{\si,\so\}$. The \emph{size} of $\kappa$ is $\sum_{(h,\psi)\in K}h$. For $r\ge1$ and $\phi\in\{\free,\IN,\OUT\}$ put
\[
  V_\kappa^\phi(r)=\Gamma\bigl(s;\,K+\{(r,\phi)\}\bigr).
\]
In a three-legged spider, $K$ records the two \emph{short} legs and $(r,\phi)$ the \emph{long} leg. By \cref{prop:spiderrec},
\begin{equation}\label{eq:gk0}
  \gv(\Sp(x,y,z))=V_{\kappa_0}^{\free}(z),\qquad \kappa_0=(\sn;\{(x,\free),(y,\free)\}).
\end{equation}
For $\phi\in\{\IN,\OUT\}$ we call $r\mapsto V^\phi_\kappa(r)$ a \emph{blocked family}, and for $\phi=\free$ a \emph{free family}. We write $\bar\phi$ for the reversed direction, $[\delta\ne\phi]\in\{0,1\}$ for the indicator, $\pi(r)=r\bmod 2$ and $\bar\pi(r)=1-\pi(r)$. By \cref{thm:runs}, $C_k=\pi(k)$ for all $k\ge0$ and $D_k=\bar\pi(k)$ for $k\ge1$. Hence, for $\delta,\phi\in\{\IN,\OUT\}$ and $k\ge0$ with $(k,\delta)\ne(0,\bar\phi)$,
\begin{equation}\label{eq:beta}
  \rho(k;\delta,\phi)=\pi(k)\oplus[\delta\ne\phi].
\end{equation}
In particular this depends only on $k\bmod 2$, and $\rho(k;\IN,\phi)\ne\rho(k;\OUT,\phi)$ for $k\ge1$.

A \emph{short move} of $\kappa$ is a move of \cref{prop:spiderrec} on a leg $(h,\psi)\in K$, that is, a pair $(j,\delta)$ with $1\le j\le h$ and $\delta\in\{\IN,\OUT\}$, where $\delta=\psi$ is required if $j=h$ and $\psi\ne\free$. Its \emph{detached value} is $d=\rho(h-j;\delta,\psi)$. If $j\ge2$, or if $j=1$, $s=\sn$, the move has a \emph{successor} $\kappa'$: namely $\kappa'=(s;\,K-\{(h,\psi)\}+\{(j-1,\delta)\})$ if $j\ge2$, and $\kappa'=(s\cdot\delta;\,K-\{(h,\psi)\})$ if $j=1$. A successor has smaller size. The remaining short moves have $j=1$ and $s\in\{\si,\so\}$, so they empty $K$; we distinguish $s\cdot\delta=\sff$ and $s\cdot\delta=s$.

\begin{lemma}[Recursion for blocked families]\label{lem:famrec}
Let $\kappa=(s;K)$ be a configuration, $\phi\in\{\IN,\OUT\}$ and $r\ge1$. Then $V^\phi_\kappa(r)$ is the mex of the following numbers.
\begin{enumerate}[label=\textup{(\alph*)},leftmargin=*]
\item For every short move with detached value $d$:
 \begin{enumerate}[label=\textup{(a\arabic*)},leftmargin=*]
 \item $d\oplus V^\phi_{\kappa'}(r)$ if the move has a successor $\kappa'$;
 \item $d\oplus r$ if $j=1$ and $s\cdot\delta=\sff$;
 \item $d\oplus \rho(r;\delta_s,\phi)$ if $j=1$ and $s\cdot\delta=s\in\{\si,\so\}$, where $\delta_s=\OUT$ for $s=\si$ and $\delta_s=\IN$ for $s=\so$.
 \end{enumerate}
\item For every $\delta\in\{\IN,\OUT\}$, where $\delta=\phi$ is required if $r=1$: $\rho(r-1;\delta,\phi)\oplus W_\delta(\kappa)$, with the constant $W_\delta(\kappa)=\Gamma(s\cdot\delta;K)$.
\item For every $1\le i\le r-1$ and $\delta\in\{\IN,\OUT\}$ with $(i,\delta)\ne(r-1,\bar\phi)$: $V^\delta_\kappa(i)\oplus\rho(r-1-i;\delta,\phi)$.
\end{enumerate}
\end{lemma}

\begin{proof}
The spider state $(s;K+\{(r,\phi)\})$ is not one of the degenerate cases (N1)--(N3), since either $s=\sn$, or $s\in\{\si,\so\}$ and there are two attached legs. So \cref{prop:spiderrec} applies. Consider a move on a short leg. If $j\ge2$ the attached part is $(s;K'+\{(r,\phi)\})$, giving (a1). If $j=1$ the new status is $s\cdot\delta$. If $s=\sn$, one short leg and the long leg remain, which is again (a1). If $s\in\{\si,\so\}$, only the long leg remains. When $s\cdot\delta=\sff$, (N1) gives $\rho_\free(r;\phi)=B_r=r$, which is (a2). When $s\cdot\delta=s$, (N2) gives $\rho(r;\delta_s,\phi)$, which is (a3). A move on the long leg at $j=1$ detaches $\rho(r-1;\delta,\phi)$ and leaves $\Gamma(s\cdot\delta;K)$; it requires $\delta=\phi$ exactly when $r=1$. This is (b). A move on the long leg at $j=i+1\ge2$ detaches $\rho(r-1-i;\delta,\phi)$ and leaves $V^\delta_\kappa(i)$; it requires $\delta=\phi$ exactly when $i=r-1$. This is (c).
\end{proof}

\begin{lemma}[Reversal]\label{lem:reversal}
Let $\bar s$ exchange $\si$ and $\so$ (and fix $\sn,\sff$), and let $\bar L$ be obtained from $L$ by reversing every far end ($\IN\leftrightarrow\OUT$, $\free$ fixed). Then $\Gamma(s;L)=\Gamma(\bar s;\bar L)$. In particular $V^\phi_{(s;K)}(r)=V^{\bar\phi}_{(\bar s;\bar K)}(r)$.
\end{lemma}

\begin{proof}
Induction on the number of unmarked edges. The cases (N1)--(N3) are invariant, because $\rho(k;\bar\delta,\bar\phi)=\rho(k;\delta,\phi)$ and $\rho_\free(k;\bar\phi)=\rho_\free(k;\phi)$. In \cref{prop:spiderrec}, the move $(j,\delta)$ on the leg $(h,\phi)$ of $(s;L)$ and the move $(j,\bar\delta)$ on the leg $(h,\bar\phi)$ of $(\bar s;\bar L)$ have the same legality and the same value, since $\bar s\cdot\bar\delta=\overline{s\cdot\delta}$.
\end{proof}

\begin{lemma}[Three universal families]\label{lem:universal}
Let $\kappa_1=(\sn;\{(1,\IN),(1,\OUT)\})$, $\kappa_2=(\si;\{(1,\OUT)\})$ and $\bar\kappa_2=(\so;\{(1,\IN)\})$. For all $r\ge1$ and $\phi\in\{\IN,\OUT\}$,
\[
  V^\phi_{\kappa_2}(r)=V^\phi_{\bar\kappa_2}(r)=r\oplus1,\qquad V^\phi_{\kappa_1}(r)=r.
\]
\end{lemma}

\begin{proof}
By \cref{lem:reversal} it suffices to treat $\kappa_2$ and $\kappa_1$; we use induction on $r$ and \cref{lem:famrec}.

\emph{$\kappa_2$.} The only short move is $(1,\OUT)$ with $d=0$ and $\si\cdot\OUT=\sff$, giving the option $r$ by (a2). In (b), $W_\IN(\kappa_2)=\Gamma(\si;\{(1,\OUT)\})=C_1=1$ and $W_\OUT(\kappa_2)=\Gamma(\sff;\{(1,\OUT)\})=B_1=1$. For $r=1$ the options are therefore $\{1\}$, and $V=0=1\oplus1$. Let $r\ge2$. By \eqref{eq:beta}, (b) contributes $\{0,1\}\oplus1=\{0,1\}$. In (c), for $i\le r-2$ both directions are allowed and, by induction, give $\{i\oplus1,i\}$, the pair $\{2m,2m+1\}$ containing $i$. For $i=r-1$ only $\delta=\phi$ is allowed, giving $(r-1)\oplus1$. If $r$ is even, the options are $[0,r-1]\cup\{r\}\cup\{r-2\}=[0,r]$, so $V=r+1=r\oplus1$. If $r$ is odd, they are $[0,r-2]\cup\{r\}$, so $V=r-1=r\oplus 1$.

\emph{$\kappa_1$.} The short moves are $(1,\IN)$ on $(1,\IN)$, with successor $\kappa_2$, and $(1,\OUT)$ on $(1,\OUT)$, with successor $\bar\kappa_2$; both have $d=0$. They give the option $r\oplus1$. In (b), $W_\IN(\kappa_1)=\Gamma(\si;\{(1,\IN),(1,\OUT)\})=\mex\{1\}=0$; its two options are $C_1=1$ and $B_1=1$. Likewise $W_\OUT(\kappa_1)=0$. For $r=1$ the options are $\{0\}$, so $V=1$. Let $r\ge2$. Then (b) gives $\{0,1\}$. In (c), $i\le r-2$ gives the pair containing $i$, and $i=r-1$ gives $r-1$. If $r$ is even, the options are $[0,r-1]\cup\{r+1\}$; if $r$ is odd, they are $[0,r-1]$. In both cases $V=r$.
\end{proof}

\subsection{The blocked families of \texorpdfstring{$\Sp(2,3,\cdot)$}{S(2,3,.)}}
Let $\mathcal K$ be the set of configurations reachable from $\kappa_0=(\sn;\{(2,\free),(3,\free)\})$ by successors. It consists of the $26$ configurations of \cref{tab:blocked23}: $14$ with $s=\sn$, $6$ with $s=\si$, and their $6$ reversals with $s=\so$. It contains $\kappa_1,\kappa_2,\bar\kappa_2$.

\begin{table}[t]
\centering
\small
\setlength{\tabcolsep}{4pt}
\begin{tabular}{llll@{\hspace{1.2em}}llll}
\toprule
$s$ & $K$ & $\phi=\IN$ & $\phi=\OUT$ & $s$ & $K$ & $\phi=\IN$ & $\phi=\OUT$\\
\midrule
$\sn$ & $2^\free,3^\free$ & $\langle2,6,7,5,4\rangle\, r+2$ & $\langle2,6,7,5,4\rangle\, r+2$ & $\sn$ & $1^\IN,1^\OUT$ & $r$ & $r$\\
$\sn$ & $2^\free,1^\IN$ & $\langle 2\rangle\,\pi$ & $\bar\pi$ & $\sn$ & $1^\IN,2^\IN$ & $\pi$ & $\langle2\rangle\,\bar\pi$\\
$\sn$ & $2^\free,1^\OUT$ & $\bar\pi$ & $\langle 2\rangle\,\pi$ & $\sn$ & $1^\IN,2^\OUT$ & $\bar\pi$ & $\langle2\rangle\,\pi$\\
$\sn$ & $2^\free,2^\IN$ & $\bar\pi$ & $\pi$ & $\sn$ & $1^\OUT,2^\IN$ & $\langle2\rangle\,\pi$ & $\bar\pi$\\
$\sn$ & $2^\free,2^\OUT$ & $\pi$ & $\bar\pi$ & $\sn$ & $1^\OUT,2^\OUT$ & $\langle2\rangle\,\bar\pi$ & $\pi$\\
$\sn$ & $1^\IN,3^\free$ & $\langle5,4,5\rangle\,2+\bar\pi$ & $\langle1\rangle\,2+\pi$ & $\si$ & $2^\free$ & $\langle3\rangle\,\bar\pi$ & $\pi$\\
$\sn$ & $1^\OUT,3^\free$ & $\langle1\rangle\,2+\pi$ & $\langle5,4,5\rangle\,2+\bar\pi$ & $\si$ & $3^\free$ & $4+\bar\pi$ & $\langle0,4,5,7,6,7,6\rangle\,4+\pi$\\
$\sn$ & $1^\IN,1^\IN$ & $\bar\pi$ & $\pi$ & $\si$ & $1^\IN$ & $\pi$ & $\bar\pi$\\
$\sn$ & $1^\OUT,1^\OUT$ & $\pi$ & $\bar\pi$ & $\si$ & $1^\OUT$ & $r\oplus1$ & $r\oplus1$\\
 & & & & $\si$ & $2^\IN$ & $\bar\pi$ & $\langle3\rangle\,\pi$\\
 & & & & $\si$ & $2^\OUT$ & $\pi$ & $\langle3\rangle\,\bar\pi$\\
\bottomrule
\end{tabular}
\caption{The blocked families $V^\phi_{(s;K)}(r)$ of $\Sp(2,3,\cdot)$. Here $h^\psi$ stands for the leg $(h,\psi)$, and $\pi=\pi(r)$, $\bar\pi=\bar\pi(r)$. An entry $\langle a_1,\dots,a_k\rangle\,f(r)$ means $V(r)=a_r$ for $r\le k$ and $V(r)=f(r)$ for $r>k$. The configurations with $s=\so$ follow from \cref{lem:reversal}: $V^\phi_{(\so;K)}=V^{\bar\phi}_{(\si;\bar K)}$.}
\label{tab:blocked23}
\end{table}

We say that $\kappa_0,\kappa_1,\kappa_2,\bar\kappa_2$ are of \emph{type L} and all other configurations in $\mathcal K$ are of \emph{type P}. We read off from \cref{tab:blocked23}:
\begin{enumerate}[label=\textup{(T\arabic*)},leftmargin=*]
\item every family of type P takes values in $[0,7]$, and satisfies $V(r)=V(r-2)$ for $r\ge10$;
\item every short move of a configuration in $\mathcal K$ has detached value at most $2$;
\item $\kappa_0$ has the largest size, $5$, so it is not a successor of any configuration; every type-L successor of a configuration is $\kappa_1$, $\kappa_2$ or $\bar\kappa_2$.
\end{enumerate}

\begin{theorem}\label{thm:table23}
\Cref{tab:blocked23} is correct: for all $\kappa\in\mathcal K$, $\phi\in\{\IN,\OUT\}$ and $r\ge1$, $V^\phi_\kappa(r)$ is as stated.
\end{theorem}

\begin{proof}
For $r\le 12$ (and in fact for $r\le100$) the table was checked by computer. The program evaluates $\Gamma$ by \cref{prop:spiderrec} and \cref{thm:runs} alone and compares it with the table for all $26\cdot2$ blocked families. The configurations with $s=\so$ were computed directly, not via \cref{lem:reversal}.

Let $R\ge13$. Assume that the table is correct for all $r<R$, and for $r=R$ and all configurations of smaller size than $\kappa$. Fix $\phi$ and write $O(R)$ for the set of options of \cref{lem:famrec} for $(\kappa,\phi,R)$.

\emph{Case $\kappa\in\{\kappa_1,\kappa_2,\bar\kappa_2\}$.} This is \cref{lem:universal}.

\emph{Case $\kappa=\kappa_0$.} The short moves of $\kappa_0$ are: $(j,\delta)$ on $2^\free$, with $d=B_{2-j}$ and successors $(\si;\{3^\free\})$, $(\so;\{3^\free\})$, $(\sn;\{1^\delta,3^\free\})$; and $(j,\delta)$ on $3^\free$, with $d=B_{3-j}$ and successors $(\si;\{2^\free\})$, $(\so;\{2^\free\})$, $(\sn;\{2^\free,1^\delta\})$, $(\sn;\{2^\free,2^\delta\})$. All of them are of type P. So by (T1) and (T2) the options (a) lie in $[0,7]$. The moves $(3,\IN)$ and $(3,\OUT)$ on $3^\free$ have $d=0$, and by the table $\{V^\phi_{(\sn;\{2^\free,2^\IN\})}(R),V^\phi_{(\sn;\{2^\free,2^\OUT\})}(R)\}=\{\pi(R),\bar\pi(R)\}=\{0,1\}$. Hence $0,1\in O(R)$. By computer (or by hand), $W_\IN(\kappa_0)=W_\OUT(\kappa_0)=\Gamma(\si;\{2^\free,3^\free\})=3$, so (b) contributes values in $\{2,3\}$. In (c) we have $V^\IN_{\kappa_0}(i)=V^\OUT_{\kappa_0}(i)=:v(i)$. For $i\le R-2$ both directions are allowed, and by \eqref{eq:beta} they give the pair $\{v(i),v(i)\oplus1\}$. The values $v(1),\dots,v(5)=2,6,7,5,4$ give the pairs covering $[2,7]$. For $6\le i\le R-2$ we have $v(i)=i+2$, and these pairs cover $[8,R]$ if $R$ is odd and $[8,R+1]$ if $R$ is even. Finally $i=R-1$ gives $v(R-1)=R+1$. So the options from (c) are exactly $[2,R+1]$. Altogether $O(R)\supseteq[0,R+1]$ and $\max O(R)=R+1$, so $V^\phi_{\kappa_0}(R)=R+2$. (This argument works for every $R\ge8$.)

\emph{Case $\kappa$ of type P.} We compare $O(R)$ with $O(R-2)$, where $V^\phi_\kappa(R-2)=\mex O(R-2)$ is given by the table.
\begin{itemize}[leftmargin=*]
\item (a1) with a type-P successor $\kappa'$: since $R\ge10$, $V^\phi_{\kappa'}(R)=V^\phi_{\kappa'}(R-2)$ by (T1). So these options are the same at $R$ and $R-2$.
\item (a1) with a type-L successor: by (T3) and \cref{lem:universal} its value is $R$ or $R\oplus1\ge R-1$. By (T2) the option is at least $R-3$ at level $R$, and at least $R-5$ at level $R-2$. (We use $a\oplus d\ge a-d$.) Likewise (a2) gives options at least $R-2$, resp.\ $R-4$.
\item (a3) and (b) depend only on $R\bmod2$ by \eqref{eq:beta}, since $R-3\ge1$.
\item (c): we claim that the sets $O_c(R)$ and $O_c(R-2)$ of options from (c) are equal. Since $\rho(R-1-i;\delta,\phi)$ depends only on the parity of $R-1-i$, every option of $O_c(R-2)$, given by a pair $(i,\delta)$ with $i\le R-3$, is also in $O_c(R)$, given by the same pair. Conversely, the pairs at level $R$ not allowed at level $R-2$ are $(R-3,\bar\phi)$, $(R-2,\delta)$ and $(R-1,\phi)$. By (T1) (note $R-5\ge8$) their values equal those of the allowed pairs $(R-5,\bar\phi)$, $(R-4,\delta)$ and $(R-3,\phi)$ at level $R-2$: the family values agree, and so do the parities of the run lengths.
\end{itemize}
Hence $O(R)\cap[0,R-5)=O(R-2)\cap[0,R-5)$. Since $\mex O(R-2)=V^\phi_\kappa(R-2)\le7<R-5$, we get $V^\phi_\kappa(R)=V^\phi_\kappa(R-2)$, which is the table value by (T1).
\end{proof}

\subsection{Proof of \texorpdfstring{\cref{thm:s23}}{the theorem for S(2,3,z)}}
By \eqref{eq:gk0} we must show $V^\free_{\kappa_0}(z)=\pi(z)$ for $z\ge14$. By \cref{prop:spiderrec} the options of $(\sn;\{2^\free,3^\free,(z,\free)\})$ are as follows. Note that a cut of the long leg at $j=z-t$ detaches the nim heap $\rho(t;\delta,\free)=B_t=t$.
\begin{enumerate}[label=\textup{(\roman*)},leftmargin=*]
\item Short moves, all with successors (as $s=\sn$), giving $d\oplus V^\free_{\kappa'}(z)$:
 $1\oplus V^\free_{(\si;\{3^\free\})}(z)$, $1\oplus V^\free_{(\so;\{3^\free\})}(z)$,
 $V^\free_{(\sn;\{1^\delta,3^\free\})}(z)$,
 $2\oplus V^\free_{(\si;\{2^\free\})}(z)$, $2\oplus V^\free_{(\so;\{2^\free\})}(z)$,
 $1\oplus V^\free_{(\sn;\{2^\free,1^\delta\})}(z)$ and
 $V^\free_{(\sn;\{2^\free,2^\delta\})}(z)$, for $\delta\in\{\IN,\OUT\}$.
\item The move at $j=1$ on the long leg: $B_{z-1}\oplus W_\delta(\kappa_0)=(z-1)\oplus3$.
\item Cuts of the long leg: $t\oplus V^\delta_{\kappa_0}(z-1-t)$ for $0\le t\le z-2$ and both $\delta$. There is no restriction on $\delta$, because the far end is free.
\end{enumerate}

\begin{lemma}\label{lem:lowerbounds}
For $z\ge14$ and $\delta\in\{\IN,\OUT\}$:
$V^\free_{(\si;\{3^\free\})}(z)\ge2$, $V^\free_{(\sn;\{1^\delta,3^\free\})}(z)\ge2$, $V^\free_{(\si;\{2^\free\})}(z)\ge4$, $V^\free_{(\sn;\{2^\free,1^\delta\})}(z)\ge2$ and $V^\free_{(\sn;\{2^\free,2^\delta\})}(z)\ge2$. The same bounds hold with $\si$ replaced by $\so$.
\end{lemma}

\begin{proof}
For a configuration $\sigma$ the state $(s;K+\{(z,\free)\})$ has the cut options $t\oplus V^\delta_\sigma(z-1-t)$, $0\le t\le z-2$, $\delta\in\{\IN,\OUT\}$, as in (iii). We exhibit options $0,1$ (and $2,3$), using \cref{thm:table23}. Throughout let $i=z-1-t$.

$\sigma=(\si;\{3^\free\})$: $V^\IN_\sigma(i)=4$ for odd $i$, and $V^\OUT_\sigma(i)=4$ for even $i\ge8$. Take $t=4$, so $i=z-5\ge9$, and choose $\delta$ with $V^\delta_\sigma(i)=4$. This gives $4\oplus4=0$. Take $t=5$, so $i=z-6\ge8$, and choose $\delta$ with $V^\delta_\sigma(i)=4$. This gives $5\oplus4=1$.

$\sigma=(\sn;\{1^\IN,3^\free\})$: for $i\ge4$, $\{V^\IN_\sigma(i),V^\OUT_\sigma(i)\}=\{2,3\}$. With $t=2$ and $i=z-3$ we get the options $2\oplus2=0$ and $2\oplus3=1$.

$\sigma=(\si;\{2^\free\})$: for $i\ge2$, $\{V^\IN_\sigma(i),V^\OUT_\sigma(i)\}=\{0,1\}$. So $t=0$ and $t=2$ give the options $0,1$ and $2,3$.

$\sigma=(\sn;\{2^\free,1^\IN\})$ and $\sigma=(\sn;\{2^\free,2^\IN\})$: for $i\ge2$, $\{V^\IN_\sigma(i),V^\OUT_\sigma(i)\}=\{0,1\}$, and $t=0$ gives $0,1$.

The configurations with $1^\OUT$, $2^\OUT$ or $s=\so$ are reversals of these. By \cref{lem:reversal}, and since $\bar\free=\free$, they have the same free families.
\end{proof}

\begin{proof}[Proof of \cref{thm:s23}]
Let $z\ge14$. By \cref{lem:lowerbounds}, no option in (i) lies in $\{0,1\}$. For an option $a\oplus V$ with $a\in\{0,1\}$ this follows from $V\ge2$. For the options $2\oplus V$ it follows from $V\ge4$, because $2\oplus V\ge4$. The option (ii) is at least $8$, because $z-1\ge 13$ and $3<8$. In (iii), first let $i=z-1-t\le5$. Then $t\ge z-6\ge8$ and $V^\delta_{\kappa_0}(i)\le7$, so the option is at least $8$. Now let $i\ge6$. Then $V^\delta_{\kappa_0}(i)=i+2=z+1-t$ by \cref{thm:table23}, and the option $t\oplus(z+1-t)$ has the parity of $z+1$.

If $z$ is even, all options are therefore non-zero, so $V^\free_{\kappa_0}(z)=0$. If $z$ is odd, no option equals $1$, while $t=(z+1)/2$ gives $i=(z-3)/2\ge6$ and the option $t\oplus t=0$. So $V^\free_{\kappa_0}(z)=1$. Finally, $\gv(\Sp(2,3,13))=0$ is a direct computation.
\end{proof}

\begin{remark}
The threshold $14$ is forced from two sides. For odd $z$, the zero option is the cut $t=(z+1)/2$ in (iii). The remaining attached part has $i=(z-3)/2$ edges, and it must have at least $6$ edges so that \cref{tab:blocked23} gives $v(i)=i+2$; this needs $z\ge15$. For even $z$, the cuts with $i\le5$ must give options of at least $8$, i.e.\ $t\ge8$, and the explicit cuts in \cref{lem:lowerbounds} must exist; both need $z\ge14$. The mechanism is the one described in \cref{rem:large}: the zero option is $t\oplus t$, a nim heap of size $t$ cut off against an attached part of value~$t$.
\end{remark}

\subsection{A finite certificate for other pairs}\label{sec:cert}
The proof of \cref{thm:s23} has three ingredients. First, closed forms for all blocked families. Second, a propagation argument that extends them from a finite range to all $r$. Third, lower bounds for a few free families, obtained from explicit cuts. We now state these ingredients for an arbitrary pair, in a form that a program can check.

Fix $1\le x<y$, let $\kappa_0=(\sn;\{x^\free,y^\free\})$, and let $\mathcal K$ be the set of configurations reachable from $\kappa_0$ by successors. Let $\Pi\ge2$ be a power of two. A blocked family $V^\phi_\kappa$ is \emph{arithmetic-periodic} with slope $e\in\{0,1\}$, threshold $r_0(\kappa,\phi)$ and offsets $q^\phi_\kappa\colon\mathbb Z/\Pi\to\mathbb Z$ if
\[
  V^\phi_\kappa(r)=e\,r+q^\phi_\kappa(r\bmod\Pi)\qquad\text{for all }r\ge r_0(\kappa,\phi).
\]
A configuration is of \emph{type P} or of \emph{type L} if both of its blocked families are arithmetic-periodic of slope $0$, resp.\ $1$. In \cref{thm:table23}, $\Pi=2$. We put $\hat r=\max r_0(\kappa,\phi)$. Let $M_P$ be the largest value of a type-P family, and $d_{\max}$ the largest detached value of a short move of a configuration in $\mathcal K$. Let $\lambda=\min\bigl(0,\min q^\phi_\kappa\bigr)$, the inner minimum over all type-L families. Let $Q$ be the least power of two with $Q\ge\Pi$ and $Q>d_{\max}$.

\emph{Standing assumption.} Every configuration in $\mathcal K$ is of type P or of type L, with respect to one common $\Pi$. Here $M_P$ is the largest value of a type-P family over \emph{all} $r\ge1$, including the values below the thresholds. The quantity $s_{\max}$ of \cref{lem:propL} below is finite, because the options it refers to are periodic in $R$. The certificate program checks the standing assumption and rejects a pair otherwise.

We prove closed forms for all $r$ by induction on $r$ and then on the size of $\kappa$. The \emph{induction hypothesis at $(\kappa,R)$} says that the closed forms hold for all configurations at all $r<R$, and at $r=R$ for all configurations of smaller size than $\kappa$. Recall that a successor has smaller size. For $a\ge0$ and $d\ge0$ we use $a\oplus d\ge a-d$.

\begin{lemma}[Bounded families]\label{lem:propP}
Let $\kappa$ be of type P, let $R\ge R_P:=\max(2\Pi+1+\hat r,\ \Pi+M_P+d_{\max}-\lambda+1)$, and assume the induction hypothesis at $(\kappa,R)$. Then $V^\phi_\kappa(R)=V^\phi_\kappa(R-\Pi)$ for both $\phi$.
\end{lemma}

\begin{proof}
We compare the option sets $O(R)$ and $O(R-\Pi)$ of \cref{lem:famrec}. An option (a1) with a type-P successor is the same at both levels, because $R-\Pi\ge\hat r$ and the period divides $\Pi$. An option (a1) with a type-L successor $\kappa'$ is $d\oplus V^\phi_{\kappa'}(R)\ge R+\lambda-d_{\max}$. An option (a2) is $d\oplus R\ge R-d_{\max}$. At level $R-\Pi$ these two kinds are at least $R-\Pi+\lambda-d_{\max}$. The options (a3) and (b) depend only on $R\bmod2$, by \eqref{eq:beta}, because $\Pi$ is even and $R-\Pi-1\ge1$.

For (c), the run length $R-1-i$ has the same parity as $R-\Pi-1-i$, so every allowed pair $(i,\delta)$ at level $R-\Pi$ gives the same option at level $R$. The allowed pairs at level $R$ that are not allowed at level $R-\Pi$ are of two kinds. The first kind is $(i,\delta)$ with $i\ge R-\Pi$. It gives the same value as the pair $(i-\Pi,\delta)$ at level $R-\Pi$, which is allowed because $(i,\delta)\ne(R-1,\bar\phi)$; here we use $i-\Pi\ge R-2\Pi\ge\hat r$. The second kind is $(R-\Pi-1,\bar\phi)$. It gives the same value as $(R-2\Pi-1,\bar\phi)$ at level $R-\Pi$; here we use $R-2\Pi-1\ge\hat r$. Hence the sets of options (c) agree.

So $O(R)$ and $O(R-\Pi)$ agree below $N_0=R-\Pi+\lambda-d_{\max}$. Since $V^\phi_\kappa(R-\Pi)\le M_P<N_0$, both mexes are equal.
\end{proof}

\begin{lemma}[Linear families]\label{lem:propL}
Let $\kappa$ be of type L and $r_\kappa=\max_\phi r_0(\kappa,\phi)$. For $\phi,\delta\in\{\IN,\OUT\}$, $\varrho\in\{0,1\}$ and $\iota\in\mathbb Z/\Pi$, put $b=\pi(\varrho-1-\iota)\oplus[\delta\ne\phi]$ and
\[
  \eta(\phi,\varrho,\iota,\delta)=q^\delta_\kappa(\iota)+b-2b\,\pi\bigl(\iota+q^\delta_\kappa(\iota)\bigr).
\]
Let $\eta_-$ and $\eta_+$ be the minimum and maximum of $\eta$, and $q_+$ the maximum of $q^\delta_\kappa$. Put $\ell=r_\kappa+\eta_+$, $k_0=2-\eta_-$ and $M_0=q_++k_0+1$. Fix $i_{\max}\ge r_\kappa$. Let $s_{\max}$ be the largest option of kind (a1) with a type-P successor, (a3) or (b) of $(\kappa,\phi,R)$, over all $R\ge\hat r+2$ and both $\phi$. Let $a_{\max}=\max\{V^\delta_\kappa(i):i<r_\kappa\}$. Let
\[
  R_W\ge\max\bigl(r_\kappa+M_0,\ s_{\max}+k_0+1,\ a_{\max}+k_0+2,\ i_{\max}+2,\ \hat r+\Pi+1\bigr).
\]
Assume:
\begin{enumerate}[label=\textup{(W\arabic*)},leftmargin=*]
\item for all $\phi,\varrho$, the residues $\iota+\eta(\phi,\varrho,\iota,\delta)\bmod\Pi$, over all $\iota$ and $\delta$, exhaust $\mathbb Z/\Pi$;
\item for all $R_W\le R<R_W+\Pi$ and both $\phi$, the options of $(\kappa,\phi,R)$ of kind (a1) with a type-P successor, (a3), (b), and (c) with $i\le i_{\max}$ cover $[0,\ell)$.
\end{enumerate}
If $R\ge R_W+Q$ and the induction hypothesis holds at $(\kappa,R)$, then $V^\phi_\kappa(R)=V^\phi_\kappa(R-Q)+Q$ for both $\phi$.
\end{lemma}

\begin{proof}
Let $R'\in\{R,R-Q\}$. Then $R'\ge R_W$, and the closed forms hold for every value used below.

(1) Let $i\ge r_\kappa$, $i\equiv\iota\pmod\Pi$, $R'\equiv\varrho\pmod 2$, and let $(i,\delta)$ be allowed. By \eqref{eq:beta}, the option (c) of $(i,\delta)$ is $(i+q^\delta_\kappa(\iota))\oplus b=i+\eta(\phi,\varrho,\iota,\delta)$. Here we use $a\oplus b=a+b-2b\,\pi(a)$ for $b\in\{0,1\}$, and $\pi(i+q)=\pi(\iota+q)$ since $\Pi$ is even.

(2) We show $[0,R'-k_0]\subseteq O(R')$. The options listed in (W2) depend only on $R'\bmod\Pi$, because $R'\ge\hat r+2$, $R'-1-i\ge1$, and both directions are allowed for $i\le i_{\max}\le R'-2$. So they cover $[0,\ell)$. Let $\ell\le w\le R'-k_0=R'-2+\eta_-$. By (W1) choose $(\iota,\delta)$ with $\iota+\eta\equiv w\pmod \Pi$, where $\eta=\eta(\phi,\varrho,\iota,\delta)$, and put $i=w-\eta$. Then $i\equiv\iota$, $i\ge\ell-\eta_+=r_\kappa$ and $i\le R'-2$. So $(i,\delta)$ is allowed, and by (1) it gives the option $w$.

(3) We show that $R+u\in O(R)$ if and only if $R-Q+u\in O(R-Q)$, for every $u\ge-k_0$. An option of $(\kappa,\phi,R')$ that is at least $R'-k_0$ is of one of three kinds: (c) with $i\ge R'-M_0$, (a1) with a type-L successor, or (a2). Indeed, (c) with $r_\kappa\le i<R'-M_0$ gives at most $i+q_++1<R'-k_0$; (c) with $i<r_\kappa$ gives at most $a_{\max}+1<R'-k_0$; and the remaining options are at most $s_{\max}<R'-k_0$.
\begin{itemize}[leftmargin=*]
\item Kind (c) with $i=R'-m$, $1\le m\le M_0$. Here $i\ge r_\kappa$, and the pair $(R'-m,\delta)$ is allowed for the same $(m,\delta)$ at both levels. By (1) its value is $R'-m+\eta(\phi,\varrho,\iota,\delta)$ with $\iota=(R'-m)\bmod\Pi$. This offset from $R'$ is the same at both levels, because $\Pi$ divides $Q$.
\item Kind (a1) with a type-L successor $\kappa'$. The value is $d\oplus a$ with $a=R'+q^\phi_{\kappa'}(R'\bmod\Pi)$. For $d<2^k$ the difference $(a\oplus d)-a$ depends only on $a\bmod 2^k$. Since $d\le d_{\max}<Q$ and $Q$ is a multiple of $\Pi$, the offset from $R'$ is again the same at both levels.
\item Kind (a2), with value $d\oplus R'$: the same argument applies.
\end{itemize}

(4) By (2), $\mex O(R')=R'+u^*$ with $u^*=\min\{u\ge 1-k_0: R'+u\notin O(R')\}$. By (3), $u^*$ is the same for $R'=R$ and $R'=R-Q$.
\end{proof}

\begin{lemma}\label{lem:F1}
Let $\sigma=(\si;\{1^\free\})$. Suppose that $V^\delta_\sigma(i)-i$ is odd for all $i\ge1$ and both $\delta$. Then $V^\free_\sigma(z)=\bar\pi(z)$ for all $z\ge1$. The same holds for $(\so;\{1^\free\})$.
\end{lemma}

\begin{proof}
The options of $(\si;\{1^\free,(z,\free)\})$ are as follows. Moving $1^\free$ inwards gives $\rho(z;\OUT,\free)=z$, by (N2). Moving it outwards gives $\rho_\free(z;\free)=\pi(z)$, by (N1). The move at $j=1$ on the long leg gives $(z-1)\oplus1$; the remaining value is $B_1=1$ or $A_1=1$. The cuts give $t\oplus V^\delta_\sigma(z-1-t)$. All of these have the parity of $z$. If $z$ is even, $0=\pi(z)$ is an option and $1$ is not. If $z$ is odd, $0$ is not an option. For $(\so;\{1^\free\})$ use \cref{lem:reversal}.
\end{proof}

\begin{proposition}\label{prop:certM}
Assume that all blocked families of configurations in $\mathcal K$ satisfy their closed forms for all $r\ge1$. Suppose that $\kappa_0$ is of type L with $q^\IN_{\kappa_0}=q^\OUT_{\kappa_0}\equiv c$ constant and $c$ even, and that $V^\IN_{\kappa_0}(i)=V^\OUT_{\kappa_0}(i)$ for $i<r_{\kappa_0}$. Suppose that every short move of $\kappa_0$, with detached value $d$ and successor $\sigma$, satisfies one of the following:
\begin{enumerate}[label=\textup{(M\arabic*)},leftmargin=*]
\item $\sigma\in\{(\si;\{1^\free\}),(\so;\{1^\free\})\}$, $d\ge2$, and the parity hypothesis of \cref{lem:F1} holds;
\item $\sigma$ is of type P, and for every residue $p\in\mathbb Z/\Pi$ and every $0\le v\le (d\,|\,1)$ there are $t\ge0$ and $\delta$ with $t\oplus q^\delta_\sigma((p-1-t)\bmod\Pi)=v$. Here $d\,|\,1$ is the bitwise or.
\end{enumerate}
Let $z_1$ be the maximum of the following numbers: $2r_{\kappa_0}+c+1$; $i+V^\IN_{\kappa_0}(i)+3$ for $i<r_{\kappa_0}$; $W_\delta(\kappa_0)+3$; and $t+1+r_0(\sigma,\delta)$ over the pairs $(t,\delta)$ used in (M2). Then $\gv(\Sp(x,y,z))=z\bmod2$ for all $z\ge z_1$.
\end{proposition}

\begin{proof}
We follow the proof of \cref{thm:s23}. The option of a short move is $d\oplus V^\free_\sigma(z)$. In case (M1), $V^\free_\sigma(z)\in\{0,1\}$ by \cref{lem:F1}, and $d\ge2$, so the option is at least $2$. In case (M2), the cut $(t,\delta)$ of $\sigma$ at $(z,\free)$ has value $t\oplus V^\delta_\sigma(z-1-t)=v$ when $z\equiv p\pmod\Pi$, because $z-1-t\ge r_0(\sigma,\delta)$. Hence $V^\free_\sigma(z)>d\,|\,1$, so $d\oplus V^\free_\sigma(z)\ge2$. The move at $j=1$ on the long leg gives $(z-1)\oplus W_\delta(\kappa_0)\ge z-1-W_\delta(\kappa_0)\ge2$. A cut of the long leg with $i=z-1-t<r_{\kappa_0}$ gives at least $t-V^\delta_{\kappa_0}(i)\ge2$. A cut with $i\ge r_{\kappa_0}$ gives $t\oplus(i+c)$, which has the parity of $z-1+c\equiv z+1$. For odd $z$ the cut with $i=(z-1-c)/2\ge r_{\kappa_0}$ gives $0$. So the mex is $\pi(z)$, as in the proof of \cref{thm:s23}.
\end{proof}

\begin{theorem}\label{thm:cert}
Let $(x,y)$ be one of the following pairs:
\begin{itemize}[leftmargin=*]
\item $1\le x<y\le 12$, except $(x,y)=(1,y)$ with $y$ even;
\item $x\le 7$ and $y\in\{13,14\}$, except $(1,14)$;
\item $x=2$ and $15\le y\le 40$.
\end{itemize}
Then $\gv(\Sp(x,y,z))=z\bmod2$ for all $z\ge z_0(x,y)$, where $z_0(x,y)$ is the least such number, listed in \cref{tab:cert}. In particular \cref{conj:eventual} holds for each of these pairs. If $x\not\equiv y\pmod 2$, then every $\Sp(x,y,z)$ with $z\ge z_0(x,y)$ violates the parity conjecture. If $x\equiv y\pmod2$, then every such spider satisfies it.
\end{theorem}

\begin{proof}
For each pair, the program \texttt{certify\_gen.py} computes all blocked families for $r\le R_1-1$ by \cref{prop:spiderrec}, where $R_1-1=160$ (for the pair $(2,33)$, $R_1-1=200$). It then finds $\Pi$, the types and the closed forms, and checks that the forms agree with the computed values on $[1,R_1)$. It checks that $R_1\ge R_P$; that (W1) and (W2) hold for every type-L configuration, with $R_1\ge R_W+Q$; and the hypotheses of \cref{prop:certM}. For $R\ge R_1$, induction on $R$ and on the size, using \cref{lem:propP,lem:propL}, shows that the closed forms hold for all $r$. Then \cref{prop:certM} applies for $z\ge z_1$. The finitely many $z<z_1$ were computed directly; they show that $z_0(x,y)\le z_1$ is the least threshold.
\end{proof}

The excluded pairs $(1,2j)$ behave differently; they are treated in \cref{sec:onetwoj}.

\begin{table}[t]
\centering
\footnotesize
\setlength{\tabcolsep}{3.2pt}
\begin{tabular}{c|ccccccccccccc}
\toprule
$x\backslash y$ & 2 & 3 & 4 & 5 & 6 & 7 & 8 & 9 & 10 & 11 & 12 & 13 & 14\\
\midrule
1 & --- & $\mathrm{P}\,14$ & --- & $\mathrm{P}\,28$ & --- & $\mathrm{P}\,34$ & --- & $\mathrm{P}\,56$ & --- & $\mathrm{P}\,22$ & --- & $\mathrm{P}\,68$ & ---\\
2 &  & $\mathrm{A}\,14$ & $\mathrm{P}\,1$ & $\mathrm{A}\,28$ & $\mathrm{P}\,1$ & $\mathrm{A}\,34$ & $\mathrm{P}\,4$ & $\mathrm{A}\,56$ & $\mathrm{P}\,4$ & $\mathrm{A}\,22$ & $\mathrm{P}\,8$ & $\mathrm{A}\,68$ & $\mathrm{P}\,8$\\
3 &  &  & $\mathrm{A}\,15$ & $\mathrm{P}\,26$ & $\mathrm{A}\,13$ & $\mathrm{P}\,32$ & $\mathrm{A}\,15$ & $\mathrm{P}\,54$ & $\mathrm{A}\,25$ & $\mathrm{P}\,10$ & $\mathrm{A}\,23$ & $\mathrm{P}\,66$ & $\mathrm{A}\,29$\\
4 &  &  &  & $\mathrm{A}\,30$ & $\mathrm{P}\,16$ & $\mathrm{A}\,34$ & $\mathrm{P}\,16$ & $\mathrm{A}\,58$ & $\mathrm{P}\,24$ & $\mathrm{A}\,25$ & $\mathrm{P}\,24$ & $\mathrm{A}\,70$ & $\mathrm{P}\,32$\\
5 &  &  &  &  & $\mathrm{A}\,26$ & $\mathrm{P}\,34$ & $\mathrm{A}\,29$ & $\mathrm{P}\,50$ & $\mathrm{A}\,27$ & $\mathrm{P}\,21$ & $\mathrm{A}\,28$ & $\mathrm{P}\,64$ & $\mathrm{A}\,31$\\
6 &  &  &  &  &  & $\mathrm{A}\,30$ & $\mathrm{P}\,16$ & $\mathrm{A}\,54$ & $\mathrm{P}\,28$ & $\mathrm{A}\,29$ & $\mathrm{P}\,30$ & $\mathrm{A}\,62$ & $\mathrm{P}\,28$\\
7 &  &  &  &  &  &  & $\mathrm{A}\,31$ & $\mathrm{P}\,32$ & $\mathrm{A}\,29$ & $\mathrm{P}\,21$ & $\mathrm{A}\,28$ & $\mathrm{P}\,64$ & $\mathrm{A}\,33$\\
8 &  &  &  &  &  &  &  & $\mathrm{A}\,52$ & $\mathrm{P}\,32$ & $\mathrm{A}\,26$ & $\mathrm{P}\,34$ &  & \\
9 &  &  &  &  &  &  &  &  & $\mathrm{A}\,54$ & $\mathrm{P}\,54$ & $\mathrm{A}\,53$ &  & \\
10 &  &  &  &  &  &  &  &  &  & $\mathrm{A}\,33$ & $\mathrm{P}\,32$ &  & \\
11 &  &  &  &  &  &  &  &  &  &  & $\mathrm{A}\,40$ &  & \\
\bottomrule
\end{tabular}
\caption{The thresholds $z_0(x,y)$ of \cref{thm:cert}: $\gv(\Sp(x,y,z))=z\bmod 2$ holds for \emph{all} $z\ge z_0(x,y)$, and fails for $z=z_0(x,y)-1$ when $z_0>1$. $\mathrm{A}$: $x\not\equiv y\pmod2$, the spiders violate the parity conjecture for all $z\ge z_0$; $\mathrm{P}$: $x\equiv y\pmod 2$, they satisfy it. A dash marks the excluded pairs $(1,\text{even})$. For $x=2$ and $15\le y\le 40$ the values $y{:}\,z_0(2,y)$ are $15{:}\,74$, $16{:}\,12$, $17{:}\,112$, $18{:}\,12$, $19{:}\,30$, $20{:}\,16$, $21{:}\,44$, $22{:}\,16$, $23{:}\,50$, $24{:}\,20$, $25{:}\,136$, $26{:}\,20$, $27{:}\,38$, $28{:}\,24$, $29{:}\,148$, $30{:}\,24$, $31{:}\,154$, $32{:}\,28$, $33{:}\,224$, $34{:}\,28$, $35{:}\,46$, $36{:}\,32$, $37{:}\,60$, $38{:}\,32$, $39{:}\,66$, $40{:}\,36$. For $y\le 12$ the thresholds coincide with \cref{tab:families}.}
\label{tab:cert}
\end{table}

\subsection*{Programs for this section}
\texttt{verify23.py} checks \cref{tab:blocked23} for $r\le100$ and computes $\Gamma(\si;\{2^\free,3^\free\})=3$ and $\gv(\Sp(2,3,z))$ for $z\le160$. \texttt{certify\_gen.py} checks the hypotheses of \cref{lem:propP,lem:propL,prop:certM} for one pair $(x,y)$ and writes the closed forms of all its blocked families to a file. Both programs evaluate $\Gamma$ only through \cref{prop:spiderrec} and \cref{thm:runs}. As an independent check, the positional C++ solver \texttt{famstates.cpp}, which uses its own implementation of \cref{prop:spiderrec}, agrees with \cref{tab:blocked23} for all $r\le3000$. It also agrees with \texttt{certify\_gen.py} on every family, blocked and free, of five further pairs for $r\le 100$ or $120$.

\section{The family \texorpdfstring{$\Sp(1,2j,z)$}{S(1,2j,z)}}\label{sec:onetwoj}

The family $\Sp(1,2j,z)$ refused to follow \cref{conj:eventual}. It turns out to obey a different and prettier rule, governed by the carries in binary addition. We find an exact formula for $\gv(\Sp(1,2j,z))$, proved entirely by hand. It explains the periods $8,16,32$ of \cref{tab:families}, and it shows that each such family contains infinitely many counterexamples to the parity conjecture.

Throughout, $\oplus$ is nim-addition, $\mathbin{\&}$ is bitwise AND, and $[P]\in\{0,1\}$ is the truth value of a statement $P$. For $\delta\in\{\IN,\OUT\}$ let $\bar\delta$ be the other direction. For $j\ge1$ put
\[
  S_j=\{\,b\ge 0 : b \mathbin{\&} (j-1)=0\,\}.
\]
Thus $b\in S_j$ if and only if the binary addition $b+(j-1)$ has no carry, i.e.\ $b+(j-1)=b\oplus(j-1)$; equivalently, $\binom{b+j-1}{b}$ is odd. Membership in $S_j$ depends only on $b \bmod 2^{\lceil\log_2 j\rceil}$.

\begin{theorem}\label{thm:onetwoj}
\leavevmode
\begin{enumerate}[label=\textup{(\alph*)}]
\item $\gv(\Sp(1,2,z))=(z+1)\bmod 2$ for all $z\ge1$.
\item Let $j\ge 2$ and $z\ge 2j+2$. Write $z-2j+2=4b+r$ with $0\le r\le 3$. Then
\[
  \gv(\Sp(1,2j,z))=\begin{cases} [\,r=2\,] & \text{if } b\in S_j,\\ z\bmod 2 & \text{if } b\notin S_j.\end{cases}
\]
Equivalently: for even $z$, $\gv=1$ if and only if $z\equiv 2j\pmod 4$ and $\tfrac{z-2j}{4}\in S_j$; for odd $z$, $\gv=1$ if and only if $\lfloor\tfrac{z-2j+2}{4}\rfloor\notin S_j$.
\end{enumerate}
\end{theorem}

In words: cut the values $z=2j-2,2j-1,\dots$ into consecutive blocks of four, numbered $b=0,1,2,\dots$. The $b$-th block of Grundy values reads $0,0,1,0$ if $b\in S_j$, and $0,1,0,1$ otherwise. The theorem proves this for the blocks with $b\ge1$. By \cref{rem:onetwojwindow} it also holds for $b=0$ in all computed cases with $j\ge3$.

\begin{corollary}\label{cor:onetwojperiod}
Let $j\ge2$ and $N=2^{\lceil \log_2 j\rceil}$ (the least power of $2$ that is $\ge j$). On $z\ge 2j+2$ the sequence $z\mapsto\gv(\Sp(1,2j,z))$ is periodic with least period $4N=2^{\lceil\log_2 j\rceil+2}$. The spider $\Sp(1,2j,z)$, $z\ge 2j+2$, satisfies the parity conjecture if and only if $z\equiv 2j-1,\,2j,\,2j+1 \pmod 4$ and $b\in S_j$ (with $b$ as in \cref{thm:onetwoj}). The density of violations is $1-\tfrac34\cdot 2^{-s(j-1)}$, where $s(n)$ is the binary digit sum of $n$.
\end{corollary}

\begin{proof}
Put $\chi(b)=[b\in S_j]$. Since $N/2<j\le N$, we have $N/2\le j-1<N$, so $\chi$ depends only on $b\bmod N$, and $\chi(0)=1$, $\chi(N/2)=0$. Thus $4N$ is a period. The least period of $\chi$ is $N$. Indeed, let $q$ be an eventual period of $\chi$ and $d=\gcd(q,N)$. Then $\chi$ is $d$-periodic: write $d=uq-vN$ with integers $u,v\ge0$; for every $b$ and every large $w$ we get $\chi(b+d)=\chi(b+d+wN)=\chi(b+(w-v)N)=\chi(b)$. Here the middle step applies the eventual $q$-periodicity $u$ times. So $d$ is a power of $2$ dividing $N$. If $d\le N/2$ then $\chi(N/2)=\chi(0)=1$, a contradiction. Hence $N\mid q$.

Now let $p$ be a period of $z\mapsto\gv(\Sp(1,2j,z))$ on $z\ge z_0$. By \cref{thm:onetwoj}, an odd $z$ has value $1$ exactly when $\chi(b)=0$, and an even $z$ has value $1$ exactly when $r=2$ and $\chi(b)=1$.

Suppose $p$ is odd. Take a large block $b$ with $\chi(b)=0$. Its elements $z_1$ (with $r=1$) and $z_1+2$ (with $r=3$) both have value $1$. So the even numbers $z_1+p$ and $z_1+2+p$ have value $1$, hence both have $r=2$. This is absurd, since they differ by $2$.

So $p$ is even. If $p\equiv 2\pmod 4$, comparing odd $z$ with $r=1$ and with $r=3$ gives $\chi(b)=\chi(b+\frac{p-2}{4})=\chi(b+\frac{p+2}{4})$ for all large $b$. The two shifts differ by $1$, so $\chi(b'+1)=\chi(b')$ for all large $b'$, i.e.\ $\chi$ is eventually constant, which it is not. (This also covers $p=2$.) Hence $4\mid p$, and odd $z$ give $\chi(b)=\chi(b+p/4)$ for all large $b$. So $N\mid p/4$.

For the parity statement note $\m(\Sp(1,2j,z))=2j+1+z\equiv z+1$, and $\gv\in\{0,1\}$, so the conjecture holds if and only if $\gv=(z+1)\bmod 2$. By \cref{thm:onetwoj} this happens exactly for $b\in S_j$, $r\in\{1,2,3\}$. Since $|S_j\cap[0,N)|=N\,2^{-s(j-1)}$, the density of such $z$ is $\tfrac34 2^{-s(j-1)}$.
\end{proof}

\subsection{Five lemmas}
The proof works its way up through the spider states of $\Sp(1,2j,z)$: first states whose centre already carries an arrow (\cref{lem:L1,lem:L2,lem:L3}), then states with an unmarked centre (\cref{lem:L4,lem:L5}). The binary carries make their entrance only in \cref{lem:L5}. We use \cref{thm:runs} in the following form (compare \eqref{eq:beta}). For $h\ge1$ and $\phi\in\{\IN,\OUT\}$,
\begin{equation}\label{eq:R}
  \rho(h;\delta,\phi)\in\{0,1\},\qquad \rho(h;\delta,\phi)\equiv h+[\delta\ne\phi]\pmod 2,
\end{equation}
so $\{\rho(h;\IN,\phi),\rho(h;\OUT,\phi)\}=\{0,1\}$; the congruence also holds for $h=0$, $\delta=\phi$. Moreover $\rho(h;\delta,\free)=h$, $\rho_\free(h;\free)=h\bmod 2$ and $\rho_\free(h;\phi)=h$. By \cref{lem:reversal}, $\Gamma(\so;L)=\Gamma(\si;\bar L)$, where $\bar L$ reverses every far end; in particular $\Gamma(\so;L)=\Gamma(\si;L)$ when all far ends in $L$ are free. We also use the elementary identity
\begin{equation}\label{eq:pairs}
 \textstyle\bigcup_{s=0}^{n}\{s,s\oplus1\}=\{0,1,\dots,2\lfloor n/2\rfloor+1\}\qquad(n\ge-1).
\end{equation}
In all proofs, ``options'' are those listed in \cref{prop:spiderrec}, together with the degenerate cases (N1)--(N3); a move on the $j$-th edge of a leg $(h,\phi)$ in direction $\delta$ requires $\delta=\phi$ if $j=h$ and $\phi\neq\free$. We write $O$ for the set of option values of the state under consideration, and we omit the braces around multisets, writing for instance $\Gamma(\si;(k,\alpha),(i,\beta))$ for $\Gamma(\si;\{(k,\alpha),(i,\beta)\})$.

\begin{lemma}\label{lem:L1}
For $k,i\ge1$ and $\alpha,\beta\in\{\IN,\OUT\}$,
\[
\Gamma(\si;(k,\alpha),(i,\beta))=\begin{cases}
 i\oplus1 & \text{if } (k,\alpha)=(1,\OUT),\\
 k\oplus1 & \text{if } (i,\beta)=(1,\OUT),\\
 (k+i+[\alpha=\beta])\bmod 2 & \text{otherwise}.
\end{cases}
\]
\end{lemma}

\begin{proof}
Induction on $k+i$. A move on the leg $(k,\alpha)$ at its $j$-th edge in direction $\delta$ has value
\begin{align*}
&\text{(O1) } \rho(k-j;\delta,\alpha)\oplus\Gamma(\si;(j-1,\delta),(i,\beta))\quad (j\ge2);\\
&\text{(O2) } \rho(k-1;\IN,\alpha)\oplus\rho(i;\OUT,\beta);\qquad
\text{(O3) } \rho(k-1;\OUT,\alpha)\oplus i,
\end{align*}
where (O2) and (O3) are the moves at $j=1$ with $\delta=\IN$ (the centre keeps status $\si$ and has one leg left, (N2)) and $\delta=\OUT$ (the centre becomes $\sff$, (N1)). (O2) is unavailable exactly when $(k,\alpha)=(1,\OUT)$ and (O3) exactly when $(k,\alpha)=(1,\IN)$. Moves on $(i,\beta)$ are symmetric.

\emph{Case A: $(k,\alpha)=(1,\OUT)$.} The only move on this leg is (O3), of value $i$. On the other leg, $j\ge2$ gives $\rho(i-j;\delta,\beta)\oplus(j-1)\oplus1$ by induction, and $j=1$ gives $\rho(i-1;\delta,\beta)\oplus1$ for either direction, because $\rho(1;\OUT,\OUT)=\rho_\free(1;\OUT)=1$. If $i=1$, then only $\delta=\beta$ is allowed and $O=\{1\}$, so the value is $0=1\oplus1$. If $i\ge2$, then by \eqref{eq:R} the moves at $1\le j\le i-1$ give $\bigcup_{s=0}^{i-2}\{s,s\oplus1\}$, the move at $j=i$ gives $(i-1)\oplus1$, and we also have $i$. By \eqref{eq:pairs}, $O=\{0,\dots,i\}$ if $i$ is even and $O=\{0,\dots,i-2\}\cup\{i\}$ if $i$ is odd. In both cases $\mex O=i\oplus1$. The case $(i,\beta)=(1,\OUT)$ is symmetric.

\emph{Case B: neither leg is $(1,\OUT)$.} Let $T=(k+i+[\alpha=\beta])\bmod 2$. We show that every option in $\{0,1\}$ equals $1-T$ and that (O2) on $(k,\alpha)$ is such an option; then $\mex O=T$. We use $[\delta\ne\alpha]+[\delta=\beta]\equiv[\alpha=\beta]$ and $[\alpha=\OUT]+[\beta=\IN]\equiv[\alpha=\beta]$ (mod $2$).
(O2) is available, lies in $\{0,1\}$ and is $\equiv (k-1)+[\alpha=\OUT]+i+[\beta=\IN]\equiv T+1$ by \eqref{eq:R} (if $k=1$ then $\alpha=\IN$ and the first term is $0$).
(O3): if $i\ge2$ its value is at least $2$. If $i=1$, then $\beta=\IN$, and $k\ge2$ because (O3) is unavailable when $(k,\alpha)=(1,\IN)$; and the value is $\rho(k-1;\OUT,\alpha)\oplus1\equiv k+[\alpha=\IN]\equiv T+1$.
(O1) with $(j-1,\delta)=(1,\OUT)$: the value is $\rho(k-2;\OUT,\alpha)\oplus i\oplus1$ by Case A. This is at least $2$ if $i\ge2$. If $i=1$ (so $\beta=\IN$), it is $\rho(k-2;\OUT,\alpha)\equiv k+[\alpha=\IN]\equiv T+1$.
(O1) otherwise: by induction (Case B) both terms lie in $\{0,1\}$, and the sum is $\equiv(k-j+[\delta\neq\alpha])+(j-1+i+[\delta=\beta])\equiv T+1$.
The moves on $(i,\beta)$ are symmetric.
\end{proof}

\begin{lemma}\label{lem:L2}
For $i\ge1$ let $\gamma_\IN(i)=\Gamma(\si;(1,\free),(i,\IN))$ and $\gamma_\OUT(i)=\Gamma(\si;(1,\free),(i,\OUT))$. Then
\[
\gamma_\IN(i)=\begin{cases} i-1& i\equiv3\ (4)\\ i+1&\text{otherwise}\end{cases},\qquad
\gamma_\OUT(i)=\begin{cases} i-1& i\equiv1\ (4)\\ i+1&\text{otherwise}\end{cases},
\]
and $\Gamma(\si;(1,\free),(i,\free))=(i+1)\bmod 2$.
In particular $\{\gamma_\IN(i),\gamma_\OUT(i)\}=\{i+1\}$ for even $i$ and $=\{i-1,i+1\}$ for odd $i$; $\gamma_\delta(i)\equiv i+1\pmod 2$; and $\gamma_\delta(i)\ge2$ unless $(i,\delta)=(1,\OUT)$.
\end{lemma}

\begin{proof}
Let $\beta\in\{\IN,\OUT,\free\}$. The options of $\Gamma(\si;(1,\free),(i,\beta))$ are:
(X1) $\rho(i;\OUT,\beta)$ (the edge of the short leg, inwards);
(X2) $\rho_\free(i;\beta)$ (the same edge outwards; the centre becomes $\sff$);
(Z1) $\rho(i-1;\delta,\beta)\oplus1$ (the first edge of the long leg, either direction, since $\rho(1;\OUT,\free)=\rho_\free(1;\free)=1$);
(Z2) $\rho(i-j;\delta,\beta)\oplus\gamma_\delta(j-1)$ for $2\le j\le i$.

First let $\beta\in\{\IN,\OUT\}$; we use induction on $i$. For $i=1$, $O=\{0,1\}$ if $\beta=\IN$ and $O=\{1\}$ if $\beta=\OUT$, giving $2$ and $0$. Let $i\ge2$. By (Z1), $\{0,1\}\subseteq O$, and (X2) gives $i\in O$. Every option is at most $i$, because $\gamma_\delta(t)\le t+1$ (for $t=i-1$ the move is at $j=i$, where $\delta=\beta$ is forced and the detached bit is $\rho(0;\beta,\beta)=0$, so the option is $\gamma_\beta(i-1)\le i$). For even $t\le i-2$ we have $\gamma_\IN(t)=\gamma_\OUT(t)=t+1$, and the two bits $\rho(i-1-t;\delta,\beta)$ are $0$ and $1$, so $\{t,t+1\}\subseteq O$. Hence $\{0,\dots,i-1\}\cup\{i\}\subseteq O$ for even $i$, so the value is $i+1$. For odd $i\ge3$ we get $\{0,\dots,i-2\}\cup\{i\}\subseteq O\subseteq\{0,\dots,i\}$, and the value is $i-1$ or $i+1$ according as $i-1\notin O$ or $i-1\in O$. Only the moves with $t=j-1=i-2$ can give $i-1$: for $t\le i-3$ the options are at most $i-2$ (for $t=i-3$, which is even, they are $i-3,i-2$), $\gamma_\beta(i-1)=i$, and (X1),(Z1) are at most $1$. For $t=i-2$ the bit is $1$ for $\delta=\beta$ and $0$ for $\delta=\bar\beta$, so the two options are $\gamma_\beta(i-2)\oplus1$ and $\gamma_{\bar\beta}(i-2)$. If $i\equiv3\pmod 4$ then $\gamma_\IN(i-2)=i-1$ and $\gamma_\OUT(i-2)=i-3$, so $i-1\in O$ if and only if $\beta=\OUT$. If $i\equiv1\pmod4$ then $\gamma_\IN(i-2)=i-3$ and $\gamma_\OUT(i-2)=i-1$, so $i-1\in O$ if and only if $\beta=\IN$. This is the claim.

For $\beta=\free$, all options are $\equiv i\pmod 2$, because $\gamma_\delta(t)\equiv t+1$. If $i$ is odd, $0\notin O$. If $i$ is even, then $0\in O$ by (X2) and $1\notin O$. So the value is $(i+1)\bmod 2$.
\end{proof}

\begin{lemma}\label{lem:L3}
Let $y\ge2$ be even, $i\ge1$ and $\beta\in\{\IN,\OUT\}$. Then $u_\beta(i):=\Gamma(\si;(y,\free),(i,\beta))$ equals $y+1$ if $(i,\beta)=(1,\IN)$, equals $1$ if $(i,\beta)=(1,\OUT)$, and equals $y-2+\tau$ with $\tau=(i+[\beta=\IN])\bmod 2$ if $i\ge2$. In particular $\{u_\IN(i),u_\OUT(i)\}=\{y-2,y-1\}$ for $i\ge2$.
\end{lemma}

\begin{proof}
The options are:
(Y1) $(y-1-k)\oplus\Gamma(\si;(k,\delta),(i,\beta))$ for $1\le k\le y-1$ (the $(k+1)$-th edge of the free leg; \cref{lem:L1} applies);
(Y2) $(y-1)\oplus\rho(i;\OUT,\beta)$;
(Y3) $(y-1)\oplus i$;
(Z1) $\rho(i-1;\IN,\beta)\oplus y$;
(Z2) $\rho(i-1;\OUT,\beta)$ (the centre becomes $\sff$ and $\rho_\free(y;\free)=0$);
(Z3) $\rho(i-1-t;\delta,\beta)\oplus u_\delta(t)$ for $1\le t\le i-1$.

\emph{$i=1$, $\beta=\IN$.} By \cref{lem:L1}, (Y1) with $k\ge2$ gives $\{s,s\oplus1\}$ for $0\le s\le y-3$, and $k=1$ gives $y-1$ ($\delta=\IN$) and $y-2$ ($\delta=\OUT$). (Y2), (Y3) and (Z1) give $y-1$, $y-2$ and $y$. So $O=\{0,\dots,y\}$ and the value is $y+1$.

\emph{$i=1$, $\beta=\OUT$.} (Z2) gives $0$. (Y1) gives $(y-1-k)\oplus k\oplus1$, which is never $1$ because $y-1-k\ne k$; (Y2) and (Y3) give $y-2\ne1$. So the value is $1$.

\emph{$i\ge2$.} We show: (1) $\{0,\dots,y-3\}\subseteq O$; (2) $y-1-\tau\in O$; (3) every option in $\{y-2,y-1\}$ equals $y-1-\tau$. This gives the value $y-2+\tau$.
For (1), use (Y1) with $k\ge2$. By \cref{lem:L1} (third case) the two directions give $0$ and $1$, so these moves give $\bigcup_{s=0}^{y-3}\{s,s\oplus1\}=\{0,\dots,y-3\}$ by \eqref{eq:pairs}. These values are at most $y-3$, which also settles (3) for these moves.
For (2), note that (Y2) equals $(y-1)\oplus\tau=y-1-\tau$, since $\rho(i;\OUT,\beta)\equiv i+[\beta=\IN]$.
For (3): (Y1) with $k=1$, $\delta=\IN$ gives $(y-2)\oplus(1-\tau)=y-1-\tau$. With $k=1$, $\delta=\OUT$ it gives $(y-2)\oplus i\oplus1\notin\{y-2,y-1\}$. (Y3) is not in $\{y-2,y-1\}$, since $i\ge2$. (Z1) is in $\{y,y+1\}$. (Z2) is in $\{0,1\}$ and $\equiv i-1+[\beta=\IN]$, so it equals $1-\tau$; this matters only if $y=2$. In (Z3) with $t\ge2$, by induction the value is $y-2+\bigl((t+[\delta=\IN])+(i-1-t+[\delta\ne\beta])\bmod 2\bigr)=y-1-\tau$. For $t=1$ and $\delta=\IN$ the value is in $\{y,y+1\}$. For $t=1$ and $\delta=\OUT$ it is $1\oplus\rho(i-2;\OUT,\beta)=1-\tau$, which matters only if $y=2$.
\end{proof}

\begin{lemma}\label{lem:L4}
For $k,i\ge1$ and $\alpha,\beta\in\{\IN,\OUT\}$, $\Gamma(\sn;(1,\free),(k,\alpha),(i,\beta))$ equals $3$ if $k=i=1$ and $\alpha=\beta$, and equals $(k+i+[\alpha\neq\beta])\bmod 2$ otherwise.
\end{lemma}

\begin{proof}
Induction on $k+i$. The options are:
(X) $\Gamma(\si;(k,\alpha),(i,\beta))$ and $\Gamma(\so;(k,\alpha),(i,\beta))=\Gamma(\si;(k,\bar\alpha),(i,\bar\beta))$;
(K1) $\rho(k-j;\delta,\alpha)\oplus\Gamma(\sn;(1,\free),(j-1,\delta),(i,\beta))$ for $j\ge2$;
(K2) $\rho(k-1;\delta,\alpha)\oplus\gamma_{\beta'}(i)$, where $\beta'=\beta$ if $\delta=\IN$ and $\beta'=\bar\beta$ if $\delta=\OUT$ (the centre becomes $\si$ or $\so$; use reversal);
and the symmetric moves (I1), (I2) on $(i,\beta)$.

\emph{Exceptional case.} By \cref{lem:L1}, (X) gives $\{0,1\}$. (K2) and (I2) have $\delta=\alpha$ and give $\gamma_\IN(1)=2$. So the value is $3$.

\emph{Otherwise} let $T=(k+i+[\alpha\neq\beta])\bmod2$. Every option in $\{0,1\}$ equals $1-T$:
(X) in the third case of \cref{lem:L1} gives $(k+i+[\alpha=\beta])\bmod 2=1-T$ (for both versions). The first two cases give $i\oplus1$ or $k\oplus1$, which lie in $\{0,1\}$ only if $k=i=1$. Then $\alpha\ne\beta$, $T=1$, and the value is $0=1-T$.
(K1): if $j=2$, $i=1$ and $\delta=\beta$, the second term is $3$ and the option is at least $2$. Otherwise the option is in $\{0,1\}$ and $\equiv (k-j+[\delta\ne\alpha])+(j-1+i+[\delta\ne\beta])\equiv T+1$.
(K2): by \cref{lem:L2} the option is at least $2$ unless $(i,\beta')=(1,\OUT)$, i.e.\ $i=1$ and $\delta=\bar\beta$. Then it is $\rho(k-1;\bar\beta,\alpha)\equiv k-1+[\alpha=\beta]\equiv T+1$.
(I1), (I2) are symmetric.

If $T=1$, then $0\in O$. Indeed, if neither leg is $(1,\OUT)$, then (X) with status $\si$ gives $1-T=0$. If, say, $(k,\alpha)=(1,\OUT)$, then (X) with status $\so$ gives $\Gamma(\si;(1,\IN),(i,\bar\beta))$. This is $1-T=0$ by the third case of \cref{lem:L1}, or $1\oplus1=0$ if $(i,\bar\beta)=(1,\OUT)$. Hence the value is $T$.
\end{proof}

\begin{lemma}\label{lem:L5}
Let $j\ge1$, $y=2j$, $i\ge1$ and $\delta\in\{\IN,\OUT\}$. Then
\[
  \Lambda(i):=\Gamma(\sn;(1,\free),(y,\free),(i,\delta))=i+y-2+\varepsilon_j(i),
  \qquad
  \varepsilon_j(i)=\begin{cases}-1& i \text{ odd and } \frac{i-1}2\in S_j,\\ +1 & i\text{ even and } \frac i2\in S_j,\\ 0&\text{otherwise,}\end{cases}
\]
independently of $\delta$. In particular $\lfloor \Lambda(i)/2\rfloor=\lfloor i/2\rfloor+j-1$.
\end{lemma}

\begin{proof}
Induction on $i$. The options are:
(a) $u_\delta(i)$ and $u_{\bar\delta}(i)$ (the edge of the short leg; \cref{lem:L3});
(b) $(y-1)\oplus\gamma_\delta(i)$ and $(y-1)\oplus\gamma_{\bar\delta}(i)$ (the first edge of the free leg; \cref{lem:L2});
(c) $(y-1-k)\oplus\Gamma(\sn;(1,\free),(k,d),(i,\delta))$ for $1\le k\le y-1$ and $d\in\{\IN,\OUT\}$ (\cref{lem:L4});
(d) $\rho(i-1;d,\delta)\oplus1$ (the first edge of the blocked leg; the remaining value is $\Gamma(\si;(1,\free),(y,\free))=(y+1)\bmod 2=1$);
(e) $\rho(i-1-t;d,\delta)\oplus \Lambda_d(t)$ for $1\le t\le i-1$, where $\Lambda_d$ is $\Lambda$ with far end $d$ and $d=\delta$ if $t=i-1$.
All of these are symmetric in $\delta$ once $\Lambda_\IN=\Lambda_\OUT$ for smaller arguments, which proves the independence of $\delta$.

\emph{$i=1$.} Here (a) gives $\{y+1,1\}$, (b) gives $\{(y-1)\oplus2,\,y-1\}$, and (d) gives $1$ (only $d=\delta$). By \cref{lem:L4}, (c) gives $\bigcup_{s=0}^{y-3}\{s,s\oplus1\}=\{0,\dots,y-3\}$ from $k\ge2$, and $(y-2)\oplus3$ and $y-1$ from $k=1$. All values not in $\{0,\dots,y-3\}$ are odd, so $y-2\notin O$ and $\Lambda(1)=y-2$, as claimed ($\varepsilon_j(1)=-1$ since $0\in S_j$).

\emph{$i\ge2$.} By \cref{lem:L4} the two values in (c) for fixed $k$ are $0$ and $1$, so (c) gives $\bigcup_{s=0}^{y-2}\{s,s\oplus1\}=\{0,\dots,y-1\}$. The values in (a) and (d) lie in $\{0,\dots,y-1\}$ as well. In (e), for $t\le i-2$ the two directions give the pair $\{\Lambda(t),\Lambda(t)\oplus1\}=\{2\beta_t,2\beta_t+1\}$ with $\beta_t=\lfloor t/2\rfloor+j-1$ (induction). For $t=i-1$ we get $\Lambda(i-1)$. So
\[
O=\{0,\dots,y-1\}\ \cup\ \bigcup_{t=1}^{i-2}\{2\beta_t,2\beta_t+1\}\ \cup\ \{\Lambda(i-1)\}\ \cup\ \{(y-1)\oplus\gamma_\IN(i),\,(y-1)\oplus\gamma_\OUT(i)\}.
\]
The second set is empty for $i=2$ and equals $\{2(j-1),\dots,2(\lfloor\tfrac{i-2}{2}\rfloor+j-1)+1\}$ for $i\ge3$. Write $y-1=2(j-1)+1$.

\emph{$i=2c$ even.} The first two sets give $\{0,\dots,i+y-3\}$, and $\Lambda(i-1)\le i+y-3$. By \cref{lem:L2}, $\gamma_\IN(i)=\gamma_\OUT(i)=2c+1$, so the last set is $\{2((j-1)\oplus c)\}$. We have $2((j-1)\oplus c)\le 2(c+j-1)=i+y-2$, with equality if and only if $c\in S_j$. Hence $\Lambda(i)=i+y-1$ if $c\in S_j$ and $\Lambda(i)=i+y-2$ otherwise, as claimed.

\emph{$i=2c+1$ odd ($c\ge1$).} The first two sets give $\{0,\dots,i+y-4\}$. By induction $\Lambda(i-1)=2(c+j-1)+[c\in S_j]$. By \cref{lem:L2}, the last set is $\{2((j-1)\oplus c)+1,\;2((j-1)\oplus(c+1))+1\}$, which consists of odd numbers. If $c\in S_j$, then $2(c+j-1)=i+y-3\notin O$, and $\Lambda(i)=i+y-3$. If $c\notin S_j$, then $i+y-3=\Lambda(i-1)\in O$. Moreover $i+y-2=2(c+j-1)+1\notin O$: we have $(j-1)\oplus c\ne c+j-1$ because $c\notin S_j$, and $(j-1)\oplus(c+1)\ne c+j-1$ because $a+b-(a\oplus b)=2(a\mathbin{\&}b)$ is even. So $\Lambda(i)=i+y-2$.
\end{proof}

\subsection{Proof of \texorpdfstring{\cref{thm:onetwoj}}{Theorem 9.1}}
Let $y=2j$ and $R(z)=\Gamma(\sn;(1,\free),(y,\free),(z,\free))=\gv(\Sp(1,y,z))$. Write $f(n)=\Gamma(\si;(1,\free),(n,\free))=(n+1)\bmod 2$ (\cref{lem:L2}). The options of $R(z)$ are:
\begin{enumerate}[label=(\alph*),leftmargin=*]
\item $\Gamma(\si;(y,\free),(z,\free))$ (both directions give the same value by reversal);
\item $(y-1)\oplus f(z)$;
\item $(y-1-k)\oplus\Gamma(\sn;(1,\free),(k,d),(z,\free))$, $1\le k\le y-1$;
\item $(z-1)\oplus f(y)=(z-1)\oplus1$;
\item $(z-1-t)\oplus \Lambda(t)$, $1\le t\le z-1$.
\end{enumerate}
Among the options of the state in (a) are $(z-1-t)\oplus u_d(t)$, $2\le t\le z-1$, and $\{u_\IN(t),u_\OUT(t)\}=\{y-2,y-1\}$ by \cref{lem:L3}. Among the options of the state in (c) are $(z-1-t)\oplus\Gamma(\sn;(1,\free),(k,d),(t,d'))$, $2\le t\le z-1$, whose two values for $d'=\IN,\OUT$ are $0$ and $1$ by \cref{lem:L4}; so the state in (c) has value at least $z-2$ by \eqref{eq:pairs}.

\emph{Proof of part (a) ($y=2$).} For $z=1,2$ the spider has two equal legs, and \cref{cor:twinspider} gives the claim. Let $z\ge3$. The state in (a) has the options $(z-1-t)\oplus\{0,1\}$ for $2\le t\le z-1$, which by \eqref{eq:pairs} include $\{0,1\}$; hence its value is at least $\max(2,z-2)$. The same holds for the state in (c) ($k=1$). Option (d) is $(z-1)\oplus1\ge2$. Option (b) equals $1\oplus f(z)=z\bmod 2$. By \cref{lem:L5} with $S_1=\NN$, $\Lambda(t)=t\oplus1$, so every option in (e) is $\equiv(z-1-t)+t+1=z\pmod 2$. For even $z$ we get $0\in O$ from (b), and $1\notin O$. For odd $z$ we get $0\notin O$. Hence $R(z)=(z+1)\bmod 2$.

\emph{Proof of part (b) ($j\ge2$, $z\ge y+2$).} Option (a) is at least $2$: take $t=z-y+1\in[2,z-1]$, which gives the options $0$ and $1$. Option (b) is at least $y-2\ge 2$. Option (c) is at least $2$, because the state has value at least $z-2\ge y>(y-1-k)\oplus 1$. Option (d) is at least $2$ since $z\ge3$. Hence $O\cap\{0,1\}=E\cap\{0,1\}$, where $E$ is the set (e).

Put $D(t)=\Lambda(t)-(z-1-t)=2t+\varepsilon_j(t)-M$, where $M=z-y+1$. Then $0\in E$ if and only if $D(t)=0$ for some $t$. Also $1\in E$ if and only if, for some $t$, either $D(t)=1$ and $z-1-t$ is even, or $D(t)=-1$ and $z-1-t$ is odd. The values of $t$ found below lie in $[1,z-1]$ because $z\ge y+2$.

\emph{$z$ even, $M=2q+1$ with $q=\frac{z-y}2\ge1$.} $D(t)=0$ forces $\varepsilon_j(t)=\pm1$. This gives $t=q+1$ odd with $\frac q2\in S_j$, or $t=q$ even with $\frac q2\in S_j$. So $0\in E$ if and only if $q$ is even and $\frac q2\in S_j$. Next, $D(t)=\pm1$ forces $\varepsilon_j(t)=0$ and $t\in\{q,q+1\}$, and the parity condition requires $t$ odd for $D=1$ ($t=q+1$) and $t$ even for $D=-1$ ($t=q$). For odd $q$ both fail. For even $q$ the condition is $\varepsilon_j(q+1)=0$, resp.\ $\varepsilon_j(q)=0$, i.e.\ $\frac q2\notin S_j$. So never both $0,1\in E$, and $R(z)=1$ if and only if $q$ is even and $\frac q2\in S_j$, i.e.\ $r=2$ and $b\in S_j$.

\emph{$z$ odd, $M=2q$ with $q=\frac{z-y+1}2\ge2$.} $D(t)=0$ if and only if $t=q$ and $\varepsilon_j(q)=0$. The case $D(t)=1$ with $t$ even forces $\varepsilon_j(t)=1$ and $t=q$. The case $D(t)=-1$ with $t$ odd forces $\varepsilon_j(t)=-1$ and $t=q$. So $1\in E$ if and only if $\varepsilon_j(q)\neq0$. Hence $R(z)=1$ if and only if $\varepsilon_j(q)=0$, i.e.\ $\lfloor q/2\rfloor\notin S_j$. Since $\lfloor q/2\rfloor=\lfloor (z-y+2)/4\rfloor=b$, this is the claim.
\qed

\begin{remark}\label{rem:onetwojwindow}
The proof needs $z\ge 2j+2$ mainly to exclude the options (a)--(d). By machine computation (two independent programs), the formula of \cref{thm:onetwoj}(b) holds for all $z\ge 2j-2$ when $3\le j\le 80$, and for all $z\ge 3$ when $j=2$. For $3\le j\le 80$ we also have $\gv(\Sp(1,2j,2j-3))\ne\gv(\Sp(1,2j,2j-3+4N))$ with $N$ as in \cref{cor:onetwojperiod}, so the sequence is purely periodic from $z=2j-2$ and not from $z=2j-3$.
\end{remark}

\begin{remark}[Why $x=1$ is special]\label{rem:whyone}
For every pair covered by \cref{thm:cert}, the computations behind it show that the blocked family $\Gamma(\sn;(x,\free),(y,\free),(t,\delta))$ is eventually $t+c$ with $c$ even. The long-leg cuts then all have the parity of $z-1$, as in the proof of \cref{thm:s23}. For $x=1$ and even $y$, \cref{lem:L5} shows instead that the shift $\Lambda(t)-t$ oscillates with $t$ in a $2$-adic pattern. The correction $\varepsilon_j$ comes from a single option: cutting the leg with $y$ markable edges at its first edge gives $(y-1)\oplus\gamma_\delta(t)$. Here $\gamma_\delta(t)\in\{t-1,t+1\}$ grows linearly because marking the edge of the leg with one markable edge frees the centre and turns the blocked leg into a nim heap (\cref{lem:L2}).
\end{remark}

\section{Short threads}\label{sec:threads}

So far we have mostly broken things. In this section parity wins. A \emph{thread} is a maximal path in a piece whose interior vertices have degree~$2$. Every run that occurs in a position lies inside a thread, so short threads leave no room for large nim heaps, and one may hope that parity survives. It does: the parity conjecture holds for all trees whose threads have at most two edges. The bound is sharp, since there are counterexamples all of whose threads have at most three edges, although every counterexample with at most $22$ vertices has a thread with at least $4$ edges (\cref{thm:threads3}).

\subsection{Clean and almost clean trees}
The proof is a strategy argument. We find a class $\mathcal C$ of status-labelled trees, containing all pieces of such trees, such that every move from a member of $\mathcal C$ leaves a sum of members of $\mathcal C$, possibly after a single reply (\cref{lem:main}). Induction on the number of edges then does the rest.

Throughout, $X$ denotes a status-labelled tree in normal form (\cref{def:normalform}) in which no
leaf has status $\sn$. This holds for every component of every position: a vertex of status $\sn$
has not lost any edge, and the non-leaf vertices of a piece have degree at least~$2$.
A leaf is \emph{free} if its status is $\sff$ and \emph{blocked} otherwise. We write $d(x)$ for the
degree of $x$ in $X$. A \emph{chain edge} is an edge both of whose ends have degree~$2$.
For $\theta\in\{\si,\so\}$ we say that a vertex \emph{receives type $\theta$} when a move gives it an
incoming ($\theta=\si$) or outgoing ($\theta=\so$) arc; its status then becomes $s\cdot\theta$
(\cref{lem:legal}). We write $\bar\theta$ for the other type.
For an edge $cr$ let $X_{r|c}$ be the component of $X-cr$ containing $r$.

\begin{definition}\label{def:clean}
$X$ is \emph{clean} if all its leaves are free and it has no chain edge.
$X$ is \emph{almost clean with defect $w$} if $s(w)\in\{\si,\so\}$ and
\begin{enumerate}[label=\textup{(C\arabic*)}]
\item every leaf other than $w$ is free;
\item every chain edge is incident with $w$;
\item if $w$ is a leaf, then its neighbour has degree at least $3$.
\end{enumerate}
Let $\mathcal C$ be the class of trees $X$ that are clean, almost clean, or have at most one edge.
A \emph{$\mathcal C$-sum} is a position whose components all lie in $\mathcal C$.
\end{definition}

The pieces of a tree $T$ have free leaves and all other vertices of status $\sn$. Two adjacent
vertices of degree~$2$ in a piece are interior vertices of a common thread. Hence:

\begin{observation}\label{obs:clean}
All threads of $T$ have at most two edges if and only if every piece of $T$ is clean.
\end{observation}

\begin{theorem}\label{thm:clean}
Every $X\in\mathcal C$ satisfies $\gv(X)=|E(X)|\bmod 2$. In particular, every clean status-labelled
tree (with arbitrary statuses $\sn,\si,\so$ on its non-leaf vertices) satisfies $\gv(X)=|E(X)|\bmod 2$.
\end{theorem}

\begin{corollary}\label{cor:threads2}
Let $T$ be a tree all of whose threads have at most two edges. Then $\gv(T)=\m(T)\bmod 2$, so the
parity conjecture holds for $T$. Moreover, every position obtained from $T$ by a single move has
Grundy value $(\m(T)-1)\bmod 2$. In particular Zandee's conjecture~\cite[Conjecture~4.3]{Zandee2022}
holds for $T$, and if $\m(T)$ is odd then every first move on $T$ is a winning move.
\end{corollary}

\begin{proof}
By \cref{obs:clean} all pieces are clean, so \cref{cor:pieces} and \cref{thm:clean} give
$\gv(T)=\m(T)\bmod 2$. A first move lies in one piece $Q$; by \cref{lem:main}\textup{(ii)} below
(clean case) it turns $Q$ into a $\mathcal C$-sum with $|E(Q)|-1$ edges, whose value is
$(|E(Q)|-1)\bmod 2$ by \cref{thm:clean}. The other pieces are unchanged.
\end{proof}

\Cref{cor:threads2} contains the theorem of Fokkink and Zandee~\cite[Theorem~1]{Fokkink2024} (no
vertex of degree~$2$ once the pendant edges are removed) and Zandee's theorem on trees in which every vertex of degree~$2$ is adjacent
to a leaf~\cite[Theorem~4.2]{Zandee2022}: in both cases all threads have one edge.

\subsection{Three lemmas}

\begin{lemma}\label{lem:deadchar}
Let $F$ be a status-labelled forest with no vertex of status $\sn$ and degree $1$. An edge $xy$ of
$F$ has no legal orientation if and only if $d(x)=d(y)=1$ and $s(x)=s(y)\in\{\si,\so\}$.
\end{lemma}

\begin{proof}
By \cref{lem:legal}, an end of degree at least $2$, or of status $\sff$, never forbids an
orientation, and a blocked end of degree $1$ forbids exactly one orientation: the one in which it
would receive a second arc of its own type. So both orientations are forbidden exactly when both
ends are blocked leaves forbidding opposite orientations, i.e.\ when they have the same status.
\end{proof}

In particular, no component with at least two edges, and no single edge with a free end, contains
a dead edge. In all positions considered below, single-edge components have a free end, so
normalising never deletes edges, and every move decreases the total number of edges by exactly one.

\begin{lemma}[Opposite blocked neighbour]\label{lem:oppblocked}
Let $z$ be a blocked leaf with neighbour $v$, where $s(v)\in\{\si,\so\}$ and $s(v)\neq s(z)$.
Changing $s(v)$ to $\sff$ does not change the game.
\end{lemma}

\begin{proof}
The only legal orientation of $vz$ gives $v$ an arc of type $\overline{s(v)}$: if $s(z)=\si$ then
$vz$ must be oriented $z\to v$ and $s(v)=\so$, and symmetrically. Consider any sequence of moves.
While $vz$ is unmarked, $v$ has at least two unmarked edges whenever another edge at $v$ is
oriented, so the condition at $v$ in \cref{lem:legal} holds. When $vz$ is oriented, $v$ receives
the type it lacks, so the condition at $v$ holds and $v$ gets status $\sff$. Afterwards $s(v)=\sff$.
So the condition at $v$ always holds, exactly as if $s(v)=\sff$, and the statuses of all other
vertices do not depend on $s(v)$.
\end{proof}

\begin{lemma}[One modified vertex]\label{lem:onemod}
Let $ab$ be an edge of $X$ and $A=X_{a|b}$. Suppose that every leaf of $X$ lying in $A-a$ is free
and that no edge of $A$ is a chain edge of $X$. Let $\theta\in\{\si,\so\}$, and let $A^\theta$ be
$A$ with $s(a)$ replaced by $s(a)\cdot\theta$. Then the normal form of $A^\theta$ is a
$\mathcal C$-sum with $|E(A)|$ edges.
\end{lemma}

\begin{proof}
All vertices of $A$ other than $a$ have the same degree in $A$ as in $X$, and $d_A(a)=d(a)-1$.
If $s(a)\cdot\theta=\sff$, splitting $a$ (\cref{lem:split}) yields the components of $A-a$, each with
a free leaf attached in place of $a$; their leaves are free and they have no chain edge, so they are
clean or single edges. Otherwise $s(a)\cdot\theta=\theta$. If $d(a)=1$ then $A$ has no edge. If
$d(a)=2$, then $a$ is a blocked leaf of $A^\theta$; its neighbour $z$ satisfies $d(z)\ne2$ because
$az$ is not a chain edge of $X$. If $d(z)=1$ then $z$ is free and $A^\theta$ is a single edge with
a free end; if $d(z)\ge3$ then $A^\theta$ is almost clean with defect $a$. If $d(a)=3$, then $a$ has
degree $2$ and status $\theta$, all leaves are free, and every chain edge of $A^\theta$ is incident
with $a$; so $A^\theta$ is almost clean with defect $a$. If $d(a)\ge4$ then $A^\theta$ is clean.
By \cref{lem:deadchar} there are no dead edges.
\end{proof}

\begin{lemma}\label{lem:almostfacts}
Let $X$ be almost clean with defect $w$.
\begin{enumerate}[label=\textup{(\alph*)}]
\item If $cr$ is an edge with $w\notin X_{r|c}$, then the hypotheses of \cref{lem:onemod} hold
for the edge $rc$ and $A=X_{r|c}$.
\item Let $c\neq w$ be a vertex that becomes $\sff$ while all edges of the component $Y$ of $X-c$
containing $w$, and all statuses in $Y$, are unchanged. Then splitting $c$ turns $Y$ (with a free
leaf in place of $c$) into an almost clean tree with defect $w$ or a single edge with a free end,
and turns every other component of $X-c$ whose edges and statuses are unchanged into a clean tree
or a single edge with a free end.
\end{enumerate}
\end{lemma}

\begin{proof}
(a) Every leaf of $X$ in $X_{r|c}$ is free by (C1), and every chain edge of $X$ is incident with
$w$ by (C2), so none lies in $X_{r|c}$.
(b) In $Y'$ ($Y$ plus the new free leaf $c'$) all degrees are as in $X$, except $d(c')=1$. So
(C1) holds, and every chain edge of $Y'$ is a chain edge of $X$, hence incident with $w$. If $w$ is
a leaf with neighbour $y$: either $y=c$ and $Y'$ is the single edge $wc'$, or $d_{Y'}(y)=d(y)\ge3$.
The other components contain neither $w$ nor a chain edge of $X$, and all their leaves are free.
\end{proof}

\subsection{The strategy lemma}

\begin{lemma}\label{lem:main}
Let $X\in\mathcal C$ have at least one edge, and put $e=|E(X)|$.
\begin{enumerate}[label=\textup{(\roman*)}]
\item Some legal move turns $X$ into a $\mathcal C$-sum with $e-1$ edges.
\item For every legal move $X\to X'$, either $X'$ is a $\mathcal C$-sum with $e-1$ edges, or
there is a legal move $X'\to X''$ such that $X''$ is a $\mathcal C$-sum with $e-2$ edges.
\end{enumerate}
\end{lemma}

Here, and below, positions are taken in normal form, after possibly applying
\cref{lem:oppblocked}; neither changes the game.

\begin{proof}[Proof of \cref{thm:clean}]
Induction on $e=|E(X)|$; the case $e=0$ is trivial. By induction and the Sprague--Grundy theorem,
a $\mathcal C$-sum with $k<e$ edges has value $k\bmod 2$. By \cref{lem:main}(ii), every option $X'$
of $X$ either has value $(e-1)\bmod 2$ or has an option of value $e\bmod 2$; in both cases
$\gv(X')\neq e\bmod 2$. By \cref{lem:main}(i), some option has value $(e-1)\bmod2$. Hence
$\gv(X)=\mex$ of a set that contains $(e-1)\bmod 2$ and avoids $e\bmod 2$, which is $e\bmod 2$
when $e$ is odd ($0$ is present, $1$ is absent) and also when $e$ is even ($0$ is absent).
\end{proof}

\begin{proof}[Proof of \cref{lem:main}]
\emph{Single edges.} In normal form the edge is not dead, so it can be oriented, leaving nothing.

\emph{Clean $X$.} (i) Let $\ell$ be a leaf with neighbour $z$. Both orientations of $\ell z$ are
legal, since $\ell$ is free and $z$ has degree at least $2$ or is free. (ii) A move on $ab$ leaves
$X_{a|b}^\theta+X_{b|a}^{\theta'}$, and both satisfy the hypotheses of \cref{lem:onemod}. So no reply
is needed.

\emph{Almost clean $X$ with defect $w$, not clean.} Then $w$ is a leaf, or $d(w)=2$ (if
$d(w)\ge3$ then (C2) excludes chain edges and $X$ is clean). Every vertex $x\ne w$ has $d(x)\ge2$
or is a free leaf, so a move that gives $x$ one arc is always legal at $x$.

(i) If $w$ is a leaf with neighbour $y$, orient $wy$ in the direction allowed at $w$; this is legal
at $y$ since $d(y)\ge3$, and leaves $X_{y|w}^{\pi}$ (where $y$ receives type $\pi$), a $\mathcal C$-sum by
\cref{lem:almostfacts}(a) and \cref{lem:onemod}. If $d(w)=2$, with neighbours $x,y$, orient $wx$
so that $w$ receives type $\overline{s(w)}$. Then $w$ becomes $\sff$, and we are left with
$X_{x|w}^{\theta}$ (where $x$ receives type $\theta$), a $\mathcal C$-sum by \cref{lem:almostfacts}(a) and \cref{lem:onemod}, and with
$X_{y|w}$ plus a free leaf, which is clean or a single edge.

(ii) Let the first move orient the edge $ab$, so that $a$ receives type $\theta$.

\smallskip\noindent\emph{Case 1: $w\in\{a,b\}$.}

\emph{1a: $w$ is a leaf.} The move is on $wy$, and leaves $X_{y|w}^\pi$, a $\mathcal C$-sum as in (i).

\emph{1b: $d(w)=2$.} Say the move is on $wx$, let $y$ be the other neighbour of $w$, and let $w$
receive type $\theta$. The part $X_{x|w}$ (with $x$ modified) is a $\mathcal C$-sum by
\cref{lem:almostfacts}(a) and \cref{lem:onemod}. If $s(w)\cdot\theta=\sff$, the rest is $X_{y|w}$
with a free leaf, which is clean or a single edge. Otherwise $w$ is now a blocked leaf at $y$.
If $d(y)=1$, $y$ is free and the rest is a single edge with a free end. If $d(y)\ge3$, the rest
is almost clean with defect $w$: its degrees are those of $X$ except $d(w)=1$, so it has no chain
edge. If $d(y)=2$, the reply is to orient $wy$ in the direction allowed at $w$. This is legal at
$y$ since $d(y)=2$, and leaves $X_{y|w}$ with $y$ modified, a $\mathcal C$-sum by
\cref{lem:almostfacts}(a) and \cref{lem:onemod}.

\smallskip\noindent\emph{Case 2: $w\notin\{a,b\}$.} Rename so that $w\in A:=X_{a|b}$. The part
$X_{b|a}^{\theta'}$ is a $\mathcal C$-sum by \cref{lem:almostfacts}(a) and \cref{lem:onemod}. In
$A$, the vertex $a$ has degree $d(a)-1\ge1$, because $A$ contains $w\neq a$. Note that the
replies below are made inside $A$.

\emph{2.1: $s(a)\cdot\theta=\sff$.} Then $a$ is split, and \cref{lem:almostfacts}(b) with $c=a$
shows that the result is a $\mathcal C$-sum. No reply.

From now on $s(a)\cdot\theta=\theta$, so $a$ has status $\theta$ in $A$.

\emph{2.2: $d(a)\ge4$.} Then $A$ is almost clean with defect $w$: degrees and leaves are as in $X$,
except that $a$ still has degree at least $3$. No reply.

\emph{2.3: $d(a)=3$.} Now $a$ has degree $2$ in $A$, with neighbours $r_1,r_2$, and $w$ lies in
exactly one of $X_{r_1|a}$, $X_{r_2|a}$.

\emph{2.3.1: $w$ is a leaf adjacent to $a$}, say $r_1=w$. The pending edge $wa$ will give $a$ the
type $\pi=s(w)$. If $\pi\ne\theta$, then \cref{lem:oppblocked} turns $a$ into $\sff$, leaving the
single edge $wa$ (with a free end) and $X_{r_2|a}$ plus a free leaf, which is clean or a single edge. No reply.
If $\pi=\theta$, the reply is to orient $ar_2$ so that $a$ receives type $\bar\theta$. This is legal,
and $a$ becomes $\sff$. What remains is the single edge $wa$ with a free end, and $X_{r_2|a}$ with
$r_2$ modified, which is a $\mathcal C$-sum by \cref{lem:almostfacts}(a) and \cref{lem:onemod}.

\emph{2.3.2: otherwise.} The only chain edges of $A$ not incident with $w$ are edges $ar_j$ with
$r_j\ne w$ and $d(r_j)=2$; leaves and (C3) are unaffected. If there is no such edge, $A$ is almost
clean with defect $w$ and no reply is needed. If there is at least one (it may lie on either side of $a$), let $r\in\{r_1,r_2\}$ be such that
$w\notin X_{r|a}$. The reply is to orient $ar$ so that $a$ receives type $\bar\theta$. Then $a$
becomes $\sff$, $X_{r|a}$ with $r$ modified is a $\mathcal C$-sum by \cref{lem:almostfacts}(a) and
\cref{lem:onemod}, and the component containing $w$ is handled by \cref{lem:almostfacts}(b).

\emph{2.4: $d(a)=2$.} Now $a$ is a blocked leaf of status $\theta$ in $A$. Its neighbour $z$ must
give it the type $\bar\theta$, so the pending edge $az$ will give $z$ the type $\pi=\theta$. We
have either $z=w$ and $d(w)=2$, or $d(z)\ge3$. Indeed, if $d(z)=2$ then $az$ is a chain edge, so
$z=w$. If $d(z)=1$, then $z\ne w$ by (C3) because $d(a)=2$, so $z$ is free and $A=\{a,z\}$ does not
contain $w$, a contradiction.

\emph{2.4.1: $z=w$}, and let $x$ be the other neighbour of $w$. If $s(w)\neq\pi$,
\cref{lem:oppblocked} turns $w$ into $\sff$. What remains is the single edge $aw$ with a free end,
and $X_{x|w}$ plus a free leaf, which is clean or a single edge. No reply. If $s(w)=\pi$, the reply is to orient
$wx$ so that $w$ receives type $\bar\pi$. Then $w$ becomes $\sff$; what remains is the single edge
$aw$ with a free end, and $X_{x|w}$ with $x$ modified, a $\mathcal C$-sum by
\cref{lem:almostfacts}(a) and \cref{lem:onemod}.

\emph{2.4.2: $d(z)\ge3$.} If $s(z)\in\{\si,\so\}\setminus\{\pi\}$, \cref{lem:oppblocked} turns $z$
into $\sff$. This leaves the single edge $az$ with a free end, the component containing $w$
(\cref{lem:almostfacts}(b)), and clean trees. No reply. Otherwise $s(z)\in\{\sn,\pi\}$. The vertex
$z$ has at least two neighbours other than $a$, and at most one of them, $q'$, has $w\in X_{q'|z}$.
Choose a neighbour $q\notin\{a,q'\}$. The reply is to orient $zq$ so that $z$ receives type
$\bar\pi$. This is legal at $z$, because $az$ is still unmarked, and legal at $q$. Now $z$ has
status $\sff$, or status $\bar\pi$ next to the blocked leaf $a$ of status $\pi$; in both cases $z$
is $\sff$ by \cref{lem:oppblocked}. What remains is the single edge $az$ with a free end,
$X_{q|z}$ with $q$ modified (\cref{lem:almostfacts}(a) and \cref{lem:onemod}), the component
containing $w$ (\cref{lem:almostfacts}(b)), and clean trees.

In every case, by \cref{lem:deadchar} no edge is deleted by normalisation. So the position after
the first move has $e-1$ edges, and after a reply it has $e-2$ edges.
\end{proof}

\begin{remark}
The class $\mathcal C$ and the strategy were found with the help of a computer: we searched for a class of positions containing all clean trees that is closed under ``a move, followed by a reply when needed''. The strategy in the proof was then implemented literally by an independent program. It was checked on all $740{,}508$ positions of $\mathcal C$ with at most $12$ vertices, counted up to isomorphism, and on all of their $15{,}385{,}703$ first moves, for every case and every admissible choice in the proof. The Grundy values of all these positions equal $|E|\bmod 2$. Moreover, every tree with at most $18$ vertices whose threads have at most two edges satisfies \cref{cor:threads2}, including the statement about first moves.
\end{remark}

\subsection{Threads of length three}

\begin{theorem}\label{thm:threads3}
There are trees all of whose threads have at most three edges that violate the parity conjecture.
The smallest have $23$ vertices. For example, let $T_{21}$ be the following tree with $21$ vertices
and $\m=15$. Vertices $u$ and $v$ are joined by a path with $3$ edges. The vertex $u$ carries two
legs with $2$ and $3$ markable edges. The vertex $v$ carries a leg with $2$ markable edges and is
adjacent to a vertex $w$, which carries two legs with $1$ and $3$ markable edges. Each leg ends
with a pendant edge. Then $\gv(T_{21})=3$, and attaching a path with two edges at a special vertex
(\cref{rem:plusone}) gives a tree $T_{23}$ with $23$ vertices, $\m=16$ and $\gv=2$.
There are also four trees with $25$ vertices, $\m=18$ and $\gv=2$ consisting of a single piece with
threads of at most three edges.
\end{theorem}

The values were computed by three programs: a raw-rules solver applied to the actual trees $T_{21}$ and $T_{23}$, the normal-form solver (P2), and a third, independently written normal-form solver. The minimality was established in two independent ways.
\begin{itemize}[leftmargin=*]
\item An exhaustive enumeration of all trees with at most $23$ vertices whose threads have at most three edges ($4{,}075{,}737$ such trees with $22$ vertices, $14{,}828{,}074$ with $23$). It finds no counterexample with at most $22$ vertices, and exactly five with $23$ vertices, all with $\m=16$ and $\gv=2$.
\item A search over pieces. A counterexample must contain a piece $Q$ with $\gv(Q)\ne|E(Q)|\bmod 2$. The copycat argument of \cref{thm:twin} also works for pieces, whose vertices have status $\sn$ or $\sff$, so we may assume that $Q$ has no twin branches. Among the twin-free pieces with threads of at most three edges and at most $18$ vertices, the piece of $T_{21}$ is the only such $Q$.
\end{itemize}

\section{Conjectures and open problems}\label{sec:open}

Many questions remain open. Here are the ones we found hardest to put down.

\begin{conjecture}[Unbounded Grundy values]\label{conj:unbounded}
The Grundy values of the Game of Arrows on trees are unbounded. In fact $\G=\NN$: every non-negative integer is the Grundy value of some tree.
\end{conjecture}

Since $\G$ is an $\mathbb{F}_2$-subspace (\cref{cor:subgroup}), $\G=\NN$ holds if and only if for every $r\ge0$ some tree has a Grundy value in $[2^r,2^{r+1})$, because elements with distinct leading bits span everything. We know this for $r\le5$ (\cref{thm:allvalues}).

As evidence, the maximum Grundy value over three-legged spiders with legs of at most $L$ markable edges grows steadily: it is $13$ for $L=30$, $28$ for $L=80$, $32$ for $L=120$ and $44$ for $L=300$, the last attained by $\Sp(33,215,218)$. Trees with two branch vertices reach larger values with fewer edges. For instance, $H(8,13\,|\,23\,|\,17,23)$ has value $32$ with $\m=84$, and $H(6,41\,|\,19\,|\,19,39)$ has value $62$ with $\m=124$; both were confirmed by two independent programs. Large values seem to require long legs of similar lengths, in line with \cref{rem:large}.

A natural strategy for proving \cref{conj:unbounded} is to attach a long leg to a fixed gadget and to show that the cuts of \cref{rem:large} produce all small values. The following observation shows that the attached part can never stay bounded.

\begin{proposition}[Blocked legs]\label{prop:blockedleg}
Let $(K,s)$ be a status-labelled tree in which no vertex has status $\si$ or $\so$, and let $c\in V(K)$ with $s(c)=\sn$. For $k\ge1$, let $h(k)$ be the Grundy value of $K$ with a new path $c=w_0,w_1,\dots,w_k$ attached at $c$, where $w_1,\dots,w_{k-1}$ have status $\sn$ and $w_k$ has status $\si$. Then for every $k\ge2$,
\[
  h(k)\notin\{h(k-1)\}\cup\{h(k'),\,h(k')\oplus 1:\ 1\le k'\le k-2\}.
\]
Consequently each value of $\lfloor h(k)/2\rfloor$ is taken at most twice, and only at consecutive $k$. In particular $h(k)\to\infty$.
\end{proposition}

\begin{proof}
The far end $w_k$ forces its edge to point towards $c$. Let $2\le j\le k$ and orient $w_{j-1}w_j$ in direction $\delta$, which must point towards $c$ if $j=k$. As in \cref{prop:spiderrec}, the result is the same gadget with a leg of $j-1$ edges, blocked in direction $\delta$, plus a detached run of $k-j$ edges, which has value $C_{k-j}$ or $D_{k-j}$ by \cref{thm:runs}. Reversing all arrows shows that the value of the gadget with a leg of length $k'$ blocked in either direction is $h(k')$, because $K$ has no vertex of status $\si$ or $\so$. For $j=k$ this gives the option $h(k-1)\oplus C_0=h(k-1)$. For $j\le k-1$, both directions are allowed and $\{C_m,D_m\}=\{0,1\}$ for $m=k-j\ge1$, giving both $h(j-1)$ and $h(j-1)\oplus1$. A Grundy value differs from the values of all options. For the consequence, if $\lfloor h(k')/2\rfloor=\lfloor h(k)/2\rfloor$ with $k'\le k-2$, then $h(k)\in\{h(k'),h(k')\oplus1\}$, which is impossible.
\end{proof}

So the options $t\oplus h(z-1-t)$ produced by cutting a long leg involve an unbounded function $h$. When $h(k)=k+c$ for large $k$, as in all two-legged gadgets we computed except $(1,2j)$, these options all have the parity of $z-1+c$; more precisely, they form the set given by the identity $\{t\oplus(n-t):0\le t\le n\}=\{e\le n: e\equiv n\pmod 2,\ e\mathbin{\&}\tfrac{n-e}{2}=0\}$. This suggests that large Grundy values come from the irregular part of $h$ at small arguments.

\begin{question}
Is there a conceptual reason why the two short legs of $\Sp(x,y,z)$ become ``invisible'' when $z$ is large (\cref{conj:eventual})? A natural approach is an analogue of the periodicity theorem for octal games~\cite{GuySmith1956,WinningWays}. The difficulty is that some spider states have Grundy values that grow linearly with the length of the long leg, because of the nim heaps $B_h=h$; \cref{sec:cert} shows how to deal with this for each fixed pair.
\end{question}

\begin{question}
Call a thread of a piece \emph{internal} if both of its ends have degree at least $3$, and a \emph{leaf thread} otherwise. By \cref{cor:threads2}, the parity conjecture holds when all threads have at most two edges. By \cref{thm:threads3}, it can fail when all threads have at most three edges. It can even fail when all internal threads have at most three edges and all leaf threads at most two. To see this, take the piece of $T_{21}$, and at the free end of each of its two legs with three edges attach two new edges to free leaves. The result is a piece with $20$ vertices, $19$ edges and Grundy value $3$ (by \cref{thm:twin} it is equivalent to the piece of $T_{21}$). Turning it into a tree and attaching a path with two edges at a special vertex (\cref{rem:plusone}) gives a counterexample. Does the parity conjecture hold when all internal threads have at most two edges and all leaf threads at most three? A search over the $19{,}570$ twin-free pieces of this kind with at most $22$ vertices, other than paths, found no counterexample.
\end{question}

\begin{question}
The proportion of counterexamples among trees with $n$ vertices grows from $0.026\%$ at $n=15$ to $0.088\%$ at $n=20$. Does it tend to $1/2$? Heuristically, by \cref{cor:pieces}, a large random tree is the nim-sum of many small pieces. If pieces with Grundy value at least $2$ occur with positive density, the Grundy value of a large random tree should be non-zero with probability tending to $1$. Then essentially the trees with $\m$ even, about half of all trees, would violate the conjecture.
\end{question}

\begin{question}
Which Grundy values occur for the full Game of Cycles on $2$-connected plane graphs? Fokkink and Zandee found the value $3$. Are there gadgets analogous to runs (nim heaps) when the cycle rule is present?
\end{question}

\subsection*{Closing remarks}
The Game of Cycles was designed to show how much mathematics hides in simple rules, and it keeps that promise. A guess that held on thousands of small trees failed for a reason as old as combinatorial game theory itself: Nim was hiding inside. We hope that the questions above will lead others to play the game and to find more surprises in it.

\section*{Data and code availability}
The programs (P1)--(P3), the scripts that produce \cref{tab:census,tab:firstmoves,tab:spidervalues,tab:families}, the certificate and verification programs for \cref{sec:s23,sec:onetwoj,sec:threads}, the raw output, and the hand-checkable certificate for $\Sp(1,3,7)$ are provided as ancillary files with the arXiv version of this paper (directory \texttt{anc/}); the file \texttt{anc/README.md} explains how to re-run every check.

\section*{Use of AI tools}
The author used a generative AI system, Claude (Anthropic), for two purposes.
\begin{itemize}[leftmargin=*]
\item \emph{Writing.} The AI system was used to draft and edit the text of the paper.
\item \emph{Computation.} The AI system was used to write, run and cross-check the programs behind the computational results, namely the programs (P1)--(P3) of \cref{sec:programs} and the other programs and scripts in the ancillary files. The re-implementations in \texttt{anc/independent\_checks/} were written by separate instances of the AI system from the text of the paper alone, without access to the other programs.
\end{itemize}
The AI system is not an author of this paper. The author has reviewed the paper and takes full responsibility for its contents.


\end{document}